\documentclass[11pt, a4paper]{article}

\usepackage[whole]{bxcjkjatype}

\usepackage{bm, amsmath, amsthm, amssymb, accents, comment}
\usepackage{ascmac, authblk}
\RequirePackage{amsthm,amsmath,amsfonts,amssymb}
\RequirePackage{natbib}
\RequirePackage{graphicx}\usepackage{tikz}
\usepackage{pgfplots}
\usepackage{subcaption}
\pgfplotsset{
	compat=newest, 
	cycle list name=exotic }

\newtheorem{theorem}{Theorem}[section]
\newtheorem{proposition}{Proposition}[section]
\newtheorem{lemma}{Lemma}[section]

\DeclareMathOperator*{\argmin}{arg\,min}

  \makeatletter
\def\widebar{\accentset{{\cc@style\underline{\mskip10mu}}}}
\makeatother

\usepackage{typearea}
\typearea{12}

\begin{document}
	
	\title{Double shrinkage priors for a normal mean matrix}	
	
	\author[1,2]{Takeru Matsuda}
	
	\author[1,2]{Fumiyasu Komaki}
	
	\author[3]{William E. Strawderman}
		
	\affil[1]{The University of Tokyo}
	
	\affil[2]{RIKEN Center for Brain Science} 

	\affil[3]{Rutgers University}

	\date{}
	
	\maketitle

	\begin{abstract}
		We consider estimation of a normal mean matrix under the Frobenius loss.
		Motivated by the Efron--Morris estimator, a generalization of Stein's prior has been recently developed, which is superharmonic and shrinks the singular values towards zero.
		The generalized Bayes estimator with respect to this prior is minimax and dominates the maximum likelihood estimator.
However, here we show that it is inadmissible by using Brown's condition.
		Then, we develop two types of priors that provide improved generalized Bayes estimators and examine their performance numerically.
		The proposed priors attain risk reduction by adding scalar shrinkage or column-wise shrinkage to singular value shrinkage.
Parallel results for Bayesian predictive densities are also given.
\end{abstract}

\section{Introduction}
Suppose that we have independent matrix observations $Y^{(1)},\dots,Y^{(N)} \in \mathbb{R}^{n \times p}$ whose entries are independent normal random variables $Y^{(t)}_{ij} \sim {\rm N} (M_{ij},1)$, where $M \in \mathbb{R}^{n \times p}$ is an unknown mean matrix. 
In the notation of \cite{Gupta}, this is expressed as $Y^{(t)} \sim {\rm N}_{n,p} (M, I_n, I_p)$ for $t=1,\dots,N$, where $I_k$ denotes the $k$-dimensional identity matrix.
We consider estimation of $M$ under the Frobenius loss
\begin{align*}
	l(M,\hat{M}) = \| \hat{M} - M \|_{{\rm F}}^2 = \sum_{a=1}^n \sum_{i=1}^p (\hat{M}_{ai} - M_{ai})^2.
\end{align*}
By sufficiency reduction, it suffices to consider the average ${Y} = (Y^{(1)}+\dots+Y^{(N)})/N \sim {\rm N} (M,I_n,N^{-1}I_p)$ in estimation of $M$. 
We assume $n-p-1>0$ in the following.
Note that vectorization reduces this problem to estimation of a normal mean vector ${\rm vec} (M)$ from ${\rm vec} ({Y}) \sim {\rm N}_{np} ({\rm vec}(M),N^{-1} I_{np})$ under the quadratic loss, which has been well studied in shrinkage estimation theory \citep{shr_book}.

\cite{Efron72} proposed an empirical Bayes estimator:
\begin{align}
	\hat{M}_{{\rm EM}} = Y \left( I_p-\frac{n-p-1}{N} (Y^{\top} Y)^{-1} \right). \label{EM_estimator}
\end{align}
This estimator can be viewed as a generalization of the James--Stein estimator ($p=1$) for a normal mean vector.
\cite{Efron72} showed that $\hat{M}_{{\rm EM}}$ is minimax and dominates the maximum likelihood estimator $\hat{M}=Y$ under the Frobenius loss.
Let $Y = U \Lambda V^{\top}$, $U \in \mathbb{R}^{n \times p}$, $V \in \mathbb{R}^{p \times p}$, $\Lambda = {\rm diag} (\sigma_1(Y), \ldots, \sigma_p(Y))$ be the singular value decomposition of $Y$, where $U^{\top} U = V^{\top} V = I_p$ and
$\sigma_1(Y) \geq \cdots \geq \sigma_p(Y) \geq 0$ are the singular values of $Y$.
\cite{Stein74} pointed out that $\hat{M}_{{\rm EM}}$ does not change the singular vectors but shrinks the singular values of $Y$ towards zero:
\begin{align*}
	\hat{M}_{{\rm EM}} = U \hat{\Lambda}_{{\rm EM}} V^{\top}, \quad \hat{\Lambda}_{{\rm EM}} = {\rm diag} (\sigma_1(\hat{M}_{{\rm EM}}), \ldots, \sigma_p(\hat{M}_{{\rm EM}})),
\end{align*}
where
\begin{align*}
	\sigma_i(\hat{M}_{{\rm EM}}) = \left( 1 - \frac{n-p-1}{N \sigma_i(Y)^2} \right) \sigma_i(Y), \quad i=1, \ldots, p.
\end{align*}
See \cite{Tsukuma,Yuasa1,Yuasa12} for recent developments around the Efron--Morris estimator.

As a Bayesian counterpart of $\hat{M}_{{\rm EM}}$, \cite{Matsuda} proposed a singular value shrinkage prior 
\begin{align}
	\pi_{{\rm SVS}} (M)=\det (M^{\top} M)^{-(n-p-1)/2}, \label{SVS}
\end{align} 
and showed that the generalized Bayes estimator $\hat{M}_{{\rm SVS}}$ with respect to $\pi_{{\rm SVS}}$ dominates the maximum likelihood estimator $\hat{M}=Y$ under the Frobenius loss.
This prior can be viewed as a generalization of Stein's prior $\pi(\mu)=\| \mu \|^{2-n}$ for a normal mean vector $\mu$ ($p=1$) by \cite{Stein74}.
Similarly to $\hat{M}_{{\rm EM}}$ in \eqref{EM_estimator} , $\hat{M}_{{\rm SVS}}$ shrinks the singular values towards zero.
Thus, it works well when the true matrix is close to low-rank.
See \cite{Matsuda22} and \cite{Matsuda23a} for details on the risk behavior of $\hat{M}_{{\rm EM}}$ and $\hat{M}_{{\rm SVS}}$.

In this paper, we show that the generalized Bayes estimator with respect to the singular value shrinkage prior $\pi_{{\rm SVS}}$ in \eqref{SVS} is inadmissible under the Frobenius loss.
Then, we develop two types of priors that provide improved generalized Bayes estimators asymptotically.
The first type adds scalar shrinkage while the second type adds column-wise shrinkage.
We conduct numerical experiments and confirm the effectiveness of the proposed priors in finite samples.
We also provide parallel results for Bayesian prediction as well as a similar improvement of the blockwise Stein prior, which was conjectured by \cite{Brown}.

This paper is organized as follows.
In Section~\ref{sec_inad}, we prove the inadmissibility of the generalized Bayes estimator with respect to the singular value shrinkage prior $\pi_{{\rm SVS}}$ in \eqref{SVS}.
In Sections~\ref{sec_scalar} and \ref{sec_column}, we provide two types of priors that asymptotically dominate the singular value shrinkage prior $\pi_{{\rm SVS}}$ in \eqref{SVS} by adding scalar or column-wise shrinkage, respectively. 
Numerical results are also given.
In Section~\ref{sec_pred}, we provide parallel results for Bayesian prediction.
Technical details and similar results for the blockwise Stein prior are given in the Appendix.

\section{Inadmissibility of the singular value shrinkage prior}\label{sec_inad}
Here, we  show that the generalized Bayes estimator with respect to the singular value shrinkage prior $\pi_{{\rm SVS}}$ in \eqref{SVS} is inadmissible under the Frobenius loss.
Since $N$ does not affect admissibility results, we fix $N=1$ for convenience in this section.

For estimation of a normal mean vector under the quadratic loss, \cite{Brown71} derived the following sufficient condition for inadmissibility of generalized Bayes estimators.

\begin{lemma}\citep{Brown71}\label{lem_brown}
	In estimation of $\theta$ from $Y \sim \mathrm{N}_d(\theta,I_d)$ under the quadratic loss, the generalized Bayes estimator of $\theta$ with respect to a prior $\pi(\theta)$ is inadmissible if
	\[
	\int_c^{\infty} {r^{1-d} \underline{m} (r)} {\rm d} r < \infty
	\]
	for some $c>0$, where 
	\[
	\underline{m} (r) = \int \frac{1}{m_{\pi} (y)} {\rm d} U_r(y),
	\]
	\[
	m_{\pi} (y) = \int p(y \mid \theta) \pi(\theta) {\rm d} \theta,
	\]
	\[
	p(y \mid \theta) = \frac{1}{(2 \pi)^{d/2}} \exp \left( -\frac{\| y-\theta \|^2}{2} \right),
	\]
	and $U_r$ is the uniform measure on the sphere of radius $r$ in $\mathbb{R}^d$.
\end{lemma}

After vectorization, estimation of a normal mean matrix $M$ from $Y \sim {\rm N}_{n,p} (M,I_n,I_p)$ under the Frobenius loss reduces to estimation of a normal mean vector ${\rm vec} (M)$ from ${\rm vec} (Y) \sim {\rm N}_{np} ({\rm vec}(M),I_{np})$ under the quadratic loss.
Then, by using Brown's condition in Lemma~\ref{lem_brown}, we obtain the following.

\begin{theorem}\label{th_inad}
	When $p \geq 2$, the generalized Bayes estimator with respect to $\pi_{{\rm SVS}}$ in \eqref{SVS} is inadmissible under the Frobenius loss.
\end{theorem}
\begin{proof}
	From $n-p-1>0$ and the AM-GM inequality
	\begin{align*}
		\left( \prod_{i=1}^p \sigma_i(M)^2 \right)^{1/p} \leq \frac{1}{p} \sum_{i=1}^p \sigma_i(M)^2,
	\end{align*}
	we have
	\begin{align*}
		\pi_{{\rm SVS}} (M) = \left( \prod_{i=1}^p \sigma_i(M)^2 \right)^{-(n-p-1)/2} &\geq \left( \frac{1}{p} \sum_{i=1}^p \sigma_i(M)^2 \right)^{-p(n-p-1)/2} \\
		&= A_{n,p} \| M \|_{{\rm F}}^{-p(n-p-1)},
	\end{align*}
	where $A_{n,p}=p^{p(n-p-1)/2}$.
	Therefore, 
	\begin{align}
		m_{{\rm SVS}} (Y) &= \int \pi_{{\rm SVS}} (M) p(Y \mid M) {\rm d} M \nonumber \\
		&\geq A_{n,p} \int \| M \|_{{\rm F}}^{-p(n-p-1)} p(Y \mid M) {\rm d} M \nonumber \\
		&= A_{n,p} {\rm E} [\| Y+Z \|_{{\rm F}}^{-p(n-p-1)}] \nonumber \\
		&\geq A_{n,p} {\rm E} [ (\| Y \|_{{\rm F}}+\| Z \|_{{\rm F}})^{-p(n-p-1)}], \label{mSVS}
	\end{align}
	where $Z=M-Y \sim {\rm N}_{n,p}(O,I_n, I_p)$ and we used the triangle inequality.
	As $\| Y \|_{{\rm F}} \to \infty$,
	\begin{align*}
		\| Y \|_{{\rm F}}^{p(n-p-1)} {\rm E} [ (\| Y \|_{{\rm F}}+\| Z \|_{{\rm F}})^{-p(n-p-1)}] = {\rm E} \left[ \left(1 + \frac{\| Z \|_{{\rm F}}}{\| Y \|_{{\rm F}}} \right)^{-p(n-p-1)} \right] \to 1,
	\end{align*}
	which yields
	\begin{align}
		{\rm E} [ (\| Y \|_{{\rm F}}+\| Z \|_{{\rm F}})^{-p(n-p-1)}] = O(\| Y \|_{{\rm F}}^{-p(n-p-1)} ). \label{chi}
	\end{align}
	
	Now, we apply Lemma~\ref{lem_brown} by noting that estimation of a normal mean matrix $M$ from $Y \sim {\rm N}_{n,p} (M,I_n,I_p)$ under the Frobenius loss is equivalent to estimation of a normal mean vector ${\rm vec} (M)$ from ${\rm vec} (Y) \sim {\rm N}_{np} ({\rm vec}(M),I_{np})$ under the quadratic loss.
	Let $U_r$ be the uniform measure on the sphere of radius $r$ in $\mathbb{R}^{n \times p}$, where the Frobenius norm is adopted for radius.
	Then, from \eqref{mSVS} and \eqref{chi},
	\begin{align*}
		\underline{m}_{{\rm SVS}} (r) &= \int \frac{1}{m_{{\rm SVS}} (Y)} {\rm d} U_r(Y) \leq C r^{p(n-p-1)}
	\end{align*}
	for some constant $C$.
	Therefore, since $-p^2-p+1<-1$ when $p \geq 2$,
	\[
	\int_1^{\infty} r^{1-np} \underline{m}_{{\rm SVS}} (r) {\rm d} r \leq C \int_1^{\infty} r^{-p^2-p+1}{\rm d} r < \infty.
	\]
	From Lemma~\ref{lem_brown}, it implies the inadmissibility of the generelized Bayes estimator with respect to $\pi_{{\rm SVS}}$ under the Frobenius loss.
\end{proof}

\section{Improvement by additional scalar shrinkage}\label{sec_scalar}
Here, motivated by the result of \cite{Efron76}, we develop a class of priors for which the generalized Bayes estimators asymptotically dominate that with respect to the singular value shrinkage prior $\pi_{{\rm SVS}}$ in \eqref{SVS}. 
\cite{Efron76} proved that the estimator
\begin{align}
	\hat{M}_{{\rm MEM}} = Y \left( I_p-\frac{n-p-1}{N} (Y^{\top} Y)^{-1} - \frac{p^2+p-2}{N \| Y \|_{\mathrm{F}}^2} I_p \right) \label{MEM_estimator}
\end{align}
dominates $\hat{M}_{{\rm EM}}$ in \eqref{EM_estimator} under the Frobenius loss. 
This estimator shrinks the singular values of $Y$ more strongly than $\hat{M}_{{\rm EM}}$:
\begin{align*}
	\hat{M}_{{\rm MEM}} = U \hat{\Lambda}_{{\rm MEM}} V^{\top}, \quad \hat{\Lambda}_{{\rm MEM}} = {\rm diag} (\sigma_1(\hat{M}_{{\rm MEM}}), \ldots, \sigma_p(\hat{M}_{{\rm MEM}})),
\end{align*}
where
\begin{align}
	\sigma_i (\hat{M}_{{\rm MEM}}) = \left( 1 - \frac{n-p-1}{N \sigma_i(Y)^2} - \frac{p^2+p-2}{N \| Y \|_{\mathrm{F}}^2} \right) \sigma_i(Y), \quad i=1, \ldots, p. \label{MEM_SV}
\end{align}
In other words, $\hat{M}_{{\rm MEM}}$ adds scalar shrinkage to $\hat{M}_{{\rm EM}}$.
\cite{Konno90,Konno91} showed corresponding results in the unknown covariance setting.
By extending these results, \cite{Tsukuma07} derived a general method for improving matrix mean estimators by adding scalar shrinkage. 

Motivated by $\hat{M}_{{\rm MEM}}$ in \eqref{MEM_estimator}, we construct priors by adding scalar shrinkage to $\pi_{{\rm SVS}}$ in \eqref{SVS}:
\begin{align}
	\pi_{{\rm MSVS1}} (M) = \pi_{{\rm SVS}} (M) \| M \|_{{\rm F}}^{-\gamma}, \label{MSVS1} \end{align}
where $\gamma \geq 0$.
Note that \cite{Tsukuma17} studied this type of prior in the context of Bayesian prediction.
Let
\begin{align*}
	m_{{\rm MSVS1}} (Y) = \int p(Y \mid M) \pi_{{\rm MSVS1}} (M) {\rm d} M
\end{align*}
be the marginal density of $Y$ under the prior $\pi_{{\rm MSVS1}} (M)$.

\begin{lemma}\label{lem_MSVS1}
	If $0 \leq \gamma < p^2+p$, then $m_{{\rm MSVS1}} (Y) < \infty$ for every $Y$.
\end{lemma}
\begin{proof}
	Since $m_{{\rm MSVS1}} (Y)$ is interpreted as the expectation of $\pi_{{\rm MSVS1}} (M)$ under $M \sim {\rm N}_{n,p}(Y,I_n, I_p)$, it suffices to show that $\pi_{{\rm MSVS1}} (M)$ is locally integrable at every $M$.
	
	First, consider $M \neq O$.
	Since 
	\begin{align*}
		m_{{\rm SVS}} (Y) &= \int \pi_{{\rm SVS}} (M) p(Y \mid M) {\rm d} M < \infty
	\end{align*}
	for every $Y$ from Lemma~1 of \cite{Matsuda}, $\pi_{{\rm SVS}} (M)$ is locally integrable at $M$.
	Also, $\| M \|_{{\rm F}}>c$ for some $c>0$ in a neighborhood of $M$.
	Thus, $\pi_{{\rm MSVS1}} (M) = \pi_{{\rm SVS}} (M) \| M \|_{{\rm F}}^{-\gamma}$ is locally integrable at $M$ if $\gamma \geq 0$.
	
	Next, consider $M = O$ and take its neighborhood $A=\{ M \mid \| M \|_{\mathrm{F}} \leq \varepsilon \}$ for $\varepsilon>0$.
	To evaluate the integral on $A$, we use the variable transformation  from $M$ to $(r,U)$, where $r= \| M \|_{\mathrm{F}}$ and $U=M/r$ so that $M = r U$.
	We have ${\rm d} M = r^{np-1} {\rm d} r {\rm d} U$.
	Also, from $\det (M^{\top} M)=r^{2p} \det (U^{\top} U)$,
	\begin{align*}
		\pi_{{\rm MSVS1}} (M) = r^{-p(n-p-1)-\gamma} \det (U^{\top} U)^{-(n-p-1)/2}.
	\end{align*}
	Thus,
	\begin{align*}
		&\int_A \pi_{{\rm MSVS1}} (M) {\rm d} M \\
		=& \int_0^{\varepsilon} r^{-p(n-p-1)-\gamma+np-1} {\rm d} r \int \det (U^{\top} U)^{-(n-p-1)/2} {\rm d} U \\
		=& \int_0^{\varepsilon} r^{p^2+p-\gamma-1} {\rm d} r \int \det (U^{\top} U)^{-(n-p-1)/2} {\rm d} U.
	\end{align*}
The integral with respect to $r$ is finite if $p^2+p-\gamma-1>-1$, which is equivalent to $\gamma < p^2+p$.
	The integral with respect to $U$ is finite due to the local integrability of $\pi_{\mathrm{SVS}}$, which corresponds to $\gamma=0$, at $M=O$.
	Therefore, $\pi_{{\rm MSVS1}} (M)$ is locally integrable at $M=O$ if $0 \leq \gamma<p^2+p$.
	
	Hence, $\pi_{{\rm MSVS1}} (M)$ is locally integrable at every $M$ if $0 \leq \gamma<p^2+p$.
\end{proof}

From Lemma~\ref{lem_MSVS1}, the generalized Bayes estimator with respect to $\pi_{{\rm MSVS1}}$ is well-defined when $0 \leq \gamma < p^2+p$.
We denote it by $\hat{M}_{{\rm MSVS1}}$.

\begin{theorem}\label{th_MSVS1_dom}
	For every $M$,
	\begin{align}
		N^2 ({\rm E}_M [\| \hat{M}_{{{\rm MSVS1}}}-M \|_{{\rm F}}^2] - {\rm E}_M [\| \hat{M}_{{{\rm SVS}}}-M \|_{{\rm F}}^2]) \to \frac{\gamma (\gamma-2p^2-2p+4)}{{\rm tr} (M^{\top} M)} \label{risk_diff}
	\end{align}
	as $N \to \infty$.
	Therefore, if $p \geq 2$ and $0 < \gamma < p^2+p$, then the generalized Bayes estimator with respect to $\pi_{{\rm MSVS1}}$ in \eqref{MSVS1} asymptotically dominates that with respect to $\pi_{{\rm SVS}}$ in \eqref{SVS} under the Frobenius loss.
\end{theorem}
\begin{proof}
	Let $K = M^{\top} M$ and $K^{ij}$ be the $(i,j)$th entry of $K^{-1}$.
	From
	\begin{align}
		\frac{\partial K_{jk}}{\partial M_{ai}} = \delta_{ik} M_{aj} + \delta_{ij} M_{ak}, \quad 	\frac{\partial}{\partial K_{ij}} \det K = K^{ij} \det K, \label{diffK}
	\end{align}
	we have
	\begin{align*}
		\frac{\partial}{\partial M_{ai}} \det K & =
		\sum_{j,k} \frac{\partial K_{jk}}{\partial M_{ai}} \frac{\partial}{\partial K_{jk}} \det K
		= 2 \sum_j M_{aj} K^{ij} \det K.
	\end{align*}
	Therefore,
	\begin{align}
		\frac{\partial}{\partial M_{ai}} \log \pi_{{\rm SVS}} (M) = -(n-p-1) \sum_j M_{aj} K^{ij}. \label{nabla1}
	\end{align}
	
	Let $\pi_{{\rm S}} (M) = \| M \|_{{\rm F}}^{-\gamma} = ({\rm tr} K)^{-\gamma/2}$.
	Since
	\begin{align*}
		\frac{\partial}{\partial M_{ai}} {\rm tr} K = 2 M_{ai}
	\end{align*}
	from \eqref{diffK}, we have
	\begin{align}
		\frac{\partial}{\partial M_{ai}} \log \pi_{{\rm S}} (M) &= -\gamma M_{ai} ({\rm tr} K)^{-1}, \label{nabla2} \\
		\frac{\partial^2}{\partial M_{ai}^2} \log \pi_{{\rm S}} (M) &= -\gamma ({\rm tr} K - 2 M_{ai}^2) ({\rm tr} K)^{-2}. \label{nabla3}
	\end{align}
	
	By using \eqref{nabla1}, \eqref{nabla2}, and \eqref{nabla3}, we obtain
	\begin{align*}
		{\rm tr} (\widetilde{\nabla} \log \pi_{{\rm SVS}} (M)^{\top} \widetilde{\nabla} \log \pi_{{\rm S}} (M) ) &= \gamma p (n-p-1) ({\rm tr} K)^{-1},\\
		{\rm tr} (\widetilde{\nabla} \log \pi_{{\rm S}} (M)^{\top} \widetilde{\nabla} \log \pi_{{\rm S}} (M) ) &= \gamma^2 ({\rm tr} K)^{-1},\\
		{\rm tr} (\widetilde{\Delta} \log \pi_{{\rm S}} (M)) &= -\gamma (np-2) ({\rm tr} K)^{-1},
	\end{align*}
	where we used the matrix derivative notations \eqref{matrix_grad} and \eqref{matrix_lap}.
	Therefore, from Lemma~\ref{lem_mat_est},
	\begin{align*}
		& {\rm E}_M [\| \hat{M}_{{\rm MSVS1}}-M \|_{{\rm F}}^2] - {\rm E}_M [\| \hat{M}_{{\rm SVS}}-M \|_{{\rm F}}^2] \\
		= & \frac{1}{N^2} {\rm tr}  ( 2 \widetilde{\nabla} \log \pi_{{\rm SVS}} (M)^{\top} \widetilde{\nabla} \log \pi_{{\rm S}} (M) + \widetilde{\nabla} \log \pi_{{\rm S}} (M)^{\top} \widetilde{\nabla} \log \pi_{{\rm S}} (M) + 2 \widetilde{\Delta} \log \pi_{{\rm S}} (M) ) \\
		& \quad + o(N^{-2}) \\
		= & \frac{1}{N^2} \gamma (\gamma-2p^2-2p+4) ({\rm tr} K)^{-1} + o(N^{-2}).
	\end{align*}
	Hence, we obtain \eqref{risk_diff}.
\end{proof}

From \eqref{risk_diff}, the choice $\gamma = p^2+p-2$ attains the minimum risk among $0 < \gamma < p^2+p$.
Note that $p^2+p-2$ also appears in the singular value decomposition form of the modified Efron--Morris estimator $\hat{M}_{{\rm MEM}}$ in \eqref{MEM_SV}.

Now, we examine the performance of $\pi_{{\rm MSVS1}}$ in \eqref{MSVS1}  by Monte Carlo simulation.
Figure~\ref{fig_MSVS1} plots the Frobenius risk of generalized Bayes estimators with respect to $\pi_{{\rm MSVS1}}$ in \eqref{MSVS1} with $\gamma=p^2+p-2$, $\pi_{{\rm SVS}}$ in \eqref{SVS} and $\pi_{{\rm S}}(M)=\| M \|_{{\rm F}}^{2-np}$, which is Stein's prior on the vectorization of $M$, for $n=10$, $p=3$ and $N=1$.
We computed the generalized Bayes estimators by using the random-walk Metropolis--Hastings algorithm with Gaussian proposal of variance $0.1$.
Note that the Frobenius risk of these estimators depends only on the singular values of $M$ due to the orthogonal invariance.
Similarly to the Efron--Morris estimator and $\pi_{{\rm SVS}}$, $\pi_{{\rm MSVS1}}$ works well when $M$ is close to low-rank.
Also, $\pi_{{\rm MSVS1}}$ attains large risk reduction when $M$ is close to the zero matrix like $\pi_{{\rm S}}$.
Thus, $\pi_{{\rm MSVS1}}$ has the best of both worlds.
Figure~\ref{fig_MSVS1_N10} plots the Frobenius risk for $n=10$, $p=3$ and $N=10$, computed by the random walk Metropolis--Hastings algorithm with proposal variance $0.005$.
The risk behavior is similar to Figure~\ref{fig_MSVS1}.
Figure~\ref{fig_MSVS1_highdim} plots the Frobenius risk for $n=20$, $p=3$ and $N=2$, computed by the random walk Metropolis--Hastings algorithm with proposal variance $0.01$.
Again, the risk behavior is similar to Figure~\ref{fig_MSVS1}.
Note that the value of $np/N=30$ is the same with Figure~\ref{fig_MSVS1}.

\begin{figure}[htbp]
	\centering
\begin{tikzpicture}
\tikzstyle{every node}=[]
	\begin{axis}[
		width=0.45\linewidth,
		xmax=10,xmin=0,
		ymax=30, ymin = 0,
xlabel={$\sigma_1(M)$},ylabel={Frobenius risk},
ylabel near ticks,
]
\addplot[very thick,
		filter discard warning=false, unbounded coords=discard
		] table {
         0    3.4525
1.0000    4.0824
2.0000    5.8827
3.0000    8.3670
4.0000   10.9478
5.0000   13.0296
6.0000   14.4476
7.0000   15.3444
8.0000   15.9608
9.0000   16.3938
10.0000   16.7493
		};
		\addplot[very thick, dashed,
		filter discard warning=false, unbounded coords=discard
		] table {
         0   12.5909
1.0000   12.9309
2.0000   13.8652
3.0000   14.9457
4.0000   15.9165
5.0000   16.6487
6.0000   17.1627
7.0000   17.4602
8.0000   17.6683
9.0000   17.8321
10.0000   17.9455			
		};
		\addplot[very thick, dotted,
		filter discard warning=false, unbounded coords=discard
		] table {
         0    2.8343
1.0000    3.5027
2.0000    5.4240
3.0000    8.2886
4.0000   11.6202
5.0000   14.8473
6.0000   17.6191
7.0000   19.8156
8.0000   21.5455
9.0000   22.9076
10.0000   24.0228
		};
	\end{axis}
\end{tikzpicture} 
\begin{tikzpicture}
\begin{axis}[
	width=0.45\linewidth,
	xmax=10,xmin=0,
	ymax=30, ymin = 0,
xlabel={$\sigma_2(M)$},
ylabel near ticks,
	legend pos=outer north east,
	legend entries={MSVS1,SVS,Stein},
	legend style={legend cell align=left,draw=none,fill=white,fill opacity=0.8,text opacity=1,},
	]
\addplot[very thick, 
	filter discard warning=false, unbounded coords=discard
	] table {
         0   16.7492
1.0000   17.1619
2.0000   18.2525
3.0000   19.5760
4.0000   20.6642
5.0000   21.4750
6.0000   22.0022
7.0000   22.3524
8.0000   22.5942
9.0000   22.7959
10.0000   22.9726
	};
	\addplot[very thick, dashed,
	filter discard warning=false, unbounded coords=discard
	] table {
         0   17.9524
1.0000   18.3660
2.0000   19.4099
3.0000   20.6575
4.0000   21.6377
5.0000   22.3528
6.0000   22.8099
7.0000   23.0850
8.0000   23.2630
9.0000   23.4171
10.0000   23.5403
	};
	\addplot[very thick, dotted,
	filter discard warning=false, unbounded coords=discard
	] table {
         0   24.0147
1.0000   24.0690
2.0000   24.2335
3.0000   24.5040
4.0000   24.7294
5.0000   25.0732
6.0000   25.4171
7.0000   25.7683
8.0000   26.1093
9.0000   26.4639
10.0000   26.8044
	};
\end{axis}
\end{tikzpicture} 
	\caption{Frobenius risk of generalized Bayes estimators for $n=10$, $p=3$ and $N=1$. Left: $\sigma_2=\sigma_3=0$. Right: $\sigma_1=10$, $\sigma_3=0$. solid: $\pi_{{\rm MSVS1}}$ with $\gamma=p^2+p-2$, dashed: $\pi_{{\rm SVS}}$, dotted: Stein's prior $\pi_{{\rm S}}$, Note that the minimax risk is $np/N=30$.}
	\label{fig_MSVS1}
\end{figure}
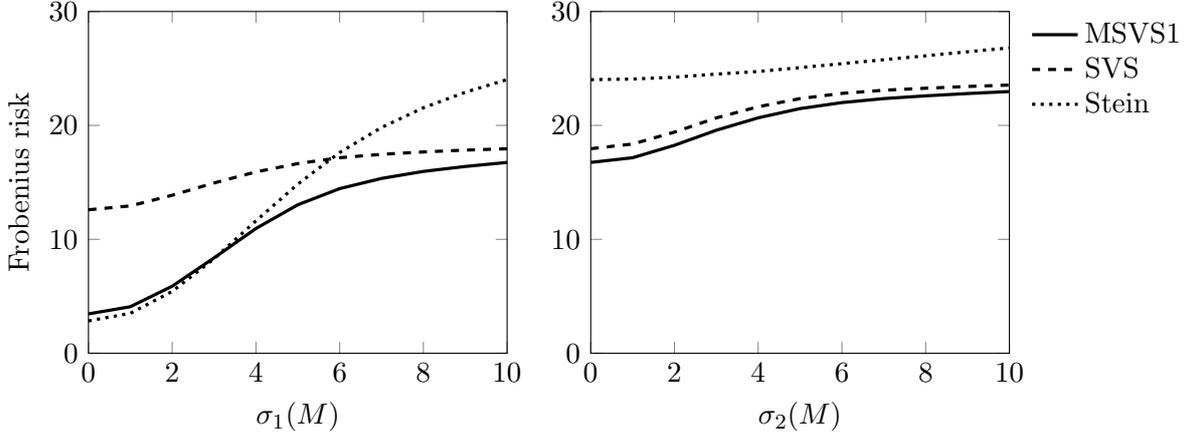
 
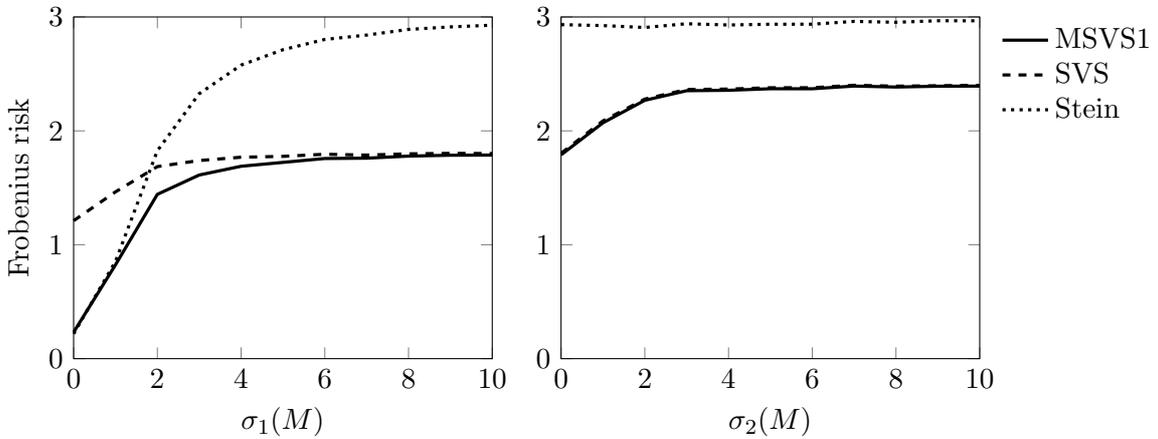
\begin{figure}[htbp]
	\centering
\begin{tikzpicture}
\tikzstyle{every node}=[]
	\begin{axis}[
		width=0.45\linewidth,
		xmax=10,xmin=0,
		ymax=3, ymin = 0,
xlabel={$\sigma_1(M)$},ylabel={Frobenius risk},
ylabel near ticks,
]
\addplot[very thick,
		filter discard warning=false, unbounded coords=discard
		] table {
         0    0.2322
1.0000    0.8241
2.0000    1.4425
3.0000    1.6123
4.0000    1.6892
5.0000    1.7232
6.0000    1.7573
7.0000    1.7602
8.0000    1.7784
9.0000    1.7861
10.0000    1.7875
		};
		\addplot[very thick, dashed,
		filter discard warning=false, unbounded coords=discard
		] table {
         0    1.2111
1.0000    1.4644
2.0000    1.6862
3.0000    1.7395
4.0000    1.7685
5.0000    1.7762
6.0000    1.7947
7.0000    1.7874
8.0000    1.7983
9.0000    1.8027
10.0000    1.8010
		};
		\addplot[very thick, dotted,
		filter discard warning=false, unbounded coords=discard
		] table {
         0    0.2178
1.0000    0.8543
2.0000    1.8251
3.0000    2.3261
4.0000    2.5772
5.0000    2.7125
6.0000    2.8027
7.0000    2.8410
8.0000    2.8905
9.0000    2.9128
10.0000    2.9277
		};
	\end{axis}
\end{tikzpicture} 
\begin{tikzpicture}
\begin{axis}[
	width=0.45\linewidth,
	xmax=10,xmin=0,
	ymax=3, ymin = 0,
xlabel={$\sigma_2(M)$},
ylabel near ticks,
	legend pos=outer north east,
	legend entries={MSVS1,SVS,Stein},
	legend style={legend cell align=left,draw=none,fill=white,fill opacity=0.8,text opacity=1,},
	]
\addplot[very thick, 
	filter discard warning=false, unbounded coords=discard
	] table {
         0    1.7893
1.0000    2.0706
2.0000    2.2683
3.0000    2.3524
4.0000    2.3562
5.0000    2.3697
6.0000    2.3695
7.0000    2.3933
8.0000    2.3847
9.0000    2.3919
10.0000    2.3931
	};
	\addplot[very thick, dashed,
	filter discard warning=false, unbounded coords=discard
	] table {
         0    1.8027
1.0000    2.0873
2.0000    2.2789
3.0000    2.3631
4.0000    2.3670
5.0000    2.3786
6.0000    2.3786
7.0000    2.4002
8.0000    2.3925
9.0000    2.3963
10.0000    2.3991
	};
	\addplot[very thick, dotted,
	filter discard warning=false, unbounded coords=discard
	] table {
         0    2.9325
1.0000    2.9242
2.0000    2.9076
3.0000    2.9404
4.0000    2.9294
5.0000    2.9361
6.0000    2.9364
7.0000    2.9611
8.0000    2.9534
9.0000    2.9667
10.0000    2.9662
	};
\end{axis}
\end{tikzpicture} 
	\caption{Frobenius risk of generalized Bayes estimators for $n=10$, $p=3$ and $N=10$. Left: $\sigma_2=\sigma_3=0$. Right: $\sigma_1=10$, $\sigma_3=0$. solid: $\pi_{{\rm MSVS1}}$ with $\gamma=p^2+p-2$, dashed: $\pi_{{\rm SVS}}$, dotted: Stein's prior $\pi_{{\rm S}}$, Note that the minimax risk is $np/N=3$.}
	\label{fig_MSVS1_N10}
\end{figure}
 
\begin{figure}[htbp]
	\centering
\begin{tikzpicture}
\tikzstyle{every node}=[]
	\begin{axis}[
		width=0.45\linewidth,
		xmax=10,xmin=0,
		ymax=30, ymin = 0,
xlabel={$\sigma_1(M)$},ylabel={Frobenius risk},
ylabel near ticks,
]
\addplot[very thick,
		filter discard warning=false, unbounded coords=discard
		] table {
         0    1.5064
1.0000    2.2765
2.0000    4.3552
3.0000    7.0364
4.0000    9.2965
5.0000   10.6595
6.0000   11.5225
7.0000   12.0963
8.0000   12.5319
9.0000   12.8742
10.0000   13.0458
		};
		\addplot[very thick, dashed,
		filter discard warning=false, unbounded coords=discard
		] table {
         0    6.2108
1.0000    6.7915
2.0000    8.1416
3.0000    9.7166
4.0000   11.0062
5.0000   11.8088
6.0000   12.4235
7.0000   12.8396
8.0000   13.1180
9.0000   13.3413
10.0000   13.4536
		};
		\addplot[very thick, dotted,
		filter discard warning=false, unbounded coords=discard
		] table {
         0    1.1970
1.0000    1.9914
2.0000    4.2412
3.0000    7.4924
4.0000   11.0413
5.0000   14.2748
6.0000   16.9695
7.0000   19.1477
8.0000   20.9436
9.0000   22.3908
10.0000   23.4651
		};
	\end{axis}
\end{tikzpicture} 
\begin{tikzpicture}
\begin{axis}[
	width=0.45\linewidth,
	xmax=10,xmin=0,
	ymax=30, ymin = 0,
xlabel={$\sigma_2(M)$},
ylabel near ticks,
	legend pos=outer north east,
	legend entries={MSVS1,SVS,Stein},
	legend style={legend cell align=left,draw=none,fill=white,fill opacity=0.8,text opacity=1,},
	]
\addplot[very thick, 
	filter discard warning=false, unbounded coords=discard
	] table {
         0   13.0943
1.0000   13.6733
2.0000   15.1695
3.0000   16.8648
4.0000   18.1859
5.0000   19.0994
6.0000   19.6036
7.0000   20.0859
8.0000   20.3277
9.0000   20.5722
10.0000   20.7644
	};
	\addplot[very thick, dashed,
	filter discard warning=false, unbounded coords=discard
	] table {
         0   13.5060
1.0000   14.0849
2.0000   15.5451
3.0000   17.1991
4.0000   18.4559
5.0000   19.3258
6.0000   19.8266
7.0000   20.2891
8.0000   20.5176
9.0000   20.7351
10.0000   20.9222
	};
	\addplot[very thick, dotted,
	filter discard warning=false, unbounded coords=discard
	] table {
         0   23.5187
1.0000   23.5464
2.0000   23.6979
3.0000   23.9532
4.0000   24.2116
5.0000   24.6052
6.0000   24.8931
7.0000   25.3754
8.0000   25.6863
9.0000   26.1049
10.0000   26.4724

	};
\end{axis}
\end{tikzpicture} 
	\caption{Frobenius risk of generalized Bayes estimators for $n=20$, $p=3$ and $N=2$. Left: $\sigma_2=\sigma_3=0$. Right: $\sigma_1=10$, $\sigma_3=0$. solid: $\pi_{{\rm MSVS1}}$ with $\gamma=p^2+p-2$, dashed: $\pi_{{\rm SVS}}$, dotted: Stein's prior $\pi_{{\rm S}}$, Note that the minimax risk is $np/N=30$.}
	\label{fig_MSVS1_highdim}
\end{figure}
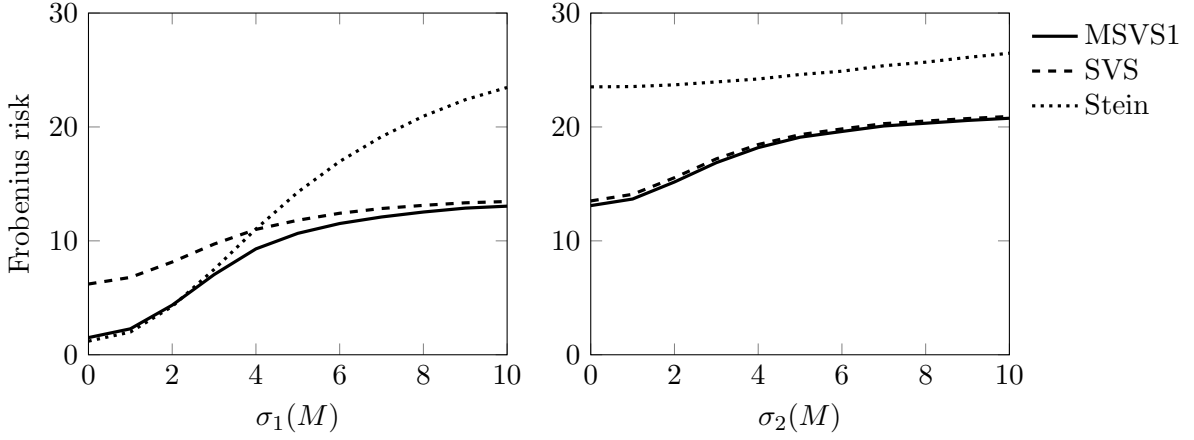
 
Improvement by additional scalar shrinkage holds even under the matrix quadratic loss \citep{Matsuda22,Matsuda23b}.

\begin{theorem}
	For every $M$,
	\begin{align}
		N^2 ({\rm E}_M [ (\hat{M}_{{{\rm MSVS1}}}-M)^{\top} &(\hat{M}_{{{\rm MSVS1}}}-M) ]- {\rm E}_M [ (\hat{M}_{{{\rm SVS}}}-M)^{\top} (\hat{M}_{{{\rm SVS}}}-M) ]) \nonumber \\
		&\to \gamma ({\rm tr} K)^{-2} (-2(p+1) ({\rm tr} K) I_p + (\gamma+4) K) \label{risk_diff_mat2}
	\end{align}
	as $N \to \infty$.
	Therefore, if $p \geq 2$ and $0 < \gamma < 2p-2$, then the generalized Bayes estimator with respect to $\pi_{{\rm MSVS1}}$ in \eqref{MSVS1} asymptotically dominates that with respect to $\pi_{{\rm SVS}}$ in \eqref{SVS} under the matrix quadratic loss.
\end{theorem}
\begin{proof}
	We use the same notation with the proof of Theorem~\ref{th_MSVS1_dom}.
	By using \eqref{nabla1}, \eqref{nabla2}, and \eqref{nabla3}, we obtain
	\begin{align*}
		(\widetilde{\nabla} \log \pi_{{\rm SVS}} (M)^{\top} \widetilde{\nabla} \log \pi_{{\rm S}} (M) ) &= \gamma (n-p-1) ({\rm tr} K)^{-1} I_p,\\
		(\widetilde{\nabla} \log \pi_{{\rm S}} (M)^{\top} \widetilde{\nabla} \log \pi_{{\rm S}} (M) ) &= \gamma^2 ({\rm tr} K)^{-2} K,\\
		(\widetilde{\Delta} \log \pi_{{\rm S}} (M)) &= -n \gamma ({\rm tr} K)^{-1} I_p +2 \gamma ({\rm tr} K)^{-2} K.
	\end{align*}
	Therefore, from Lemma~\ref{lem_mat_est2},
	\begin{align*}
		& {\rm E}_M [ (\hat{M}_{{{\rm MSVS1}}}-M)^{\top} (\hat{M}_{{{\rm MSVS1}}}-M) ] - {\rm E}_M [ (\hat{M}_{{{\rm SVS}}}-M)^{\top} (\hat{M}_{{{\rm SVS}}}-M) ] \\
		= & \frac{1}{N^2} ( 2 \widetilde{\nabla} \log \pi_{{\rm SVS}} (M)^{\top} \widetilde{\nabla} \log \pi_{{\rm S}} (M) + \widetilde{\nabla} \log \pi_{{\rm S}} (M)^{\top} \widetilde{\nabla} \log \pi_{{\rm S}} (M) + 2 \widetilde{\Delta} \log \pi_{{\rm S}} (M) ) \\
		& \quad + o(N^{-2}) \\
		= & \frac{1}{N^2} \gamma ({\rm tr} K)^{-2} (-2(p+1) ({\rm tr} K) I_p + (\gamma+4) K) + o(N^{-2}).
	\end{align*}
	Hence, we obtain \eqref{risk_diff_mat2}.
	Since $K \preceq ({\rm tr} K) I_p$ from $K \succeq O$, 
	\begin{align*}
		-2(p+1) ({\rm tr} K) I_p + (\gamma+4) K \preceq (\gamma-2p+2) ({\rm tr} K) I_p \prec O
	\end{align*}
	if $0 < \gamma < 2p-2$.
\end{proof}

The generalized Bayes estimator with respect to $\pi_{{\rm MSVS1}}$ in \eqref{MSVS1} attains minimaxity in some cases as follows.

\begin{theorem}
	If $p \geq 2$, $p+2 \leq n < 2p+2-2/p$ and $0 < \gamma \leq -np+2p^2+2p-2$, then the generalized Bayes estimator with respect to $\pi_{{\rm MSVS1}}$ in \eqref{MSVS1} is minimax under the Frobenius loss.
\end{theorem}
\begin{proof}
	From Proposition~\ref{prop_MSVS1_lap}, 
	\begin{align*}
		\Delta \pi_{{\rm MSVS1}} (M) &= \gamma (\gamma+np-2p^2-2p+2) \| M \|_{{\rm F}}^{-2} \pi_{{\rm MSVS1}} (M) \leq 0.
	\end{align*}
	Thus, $\pi_{{\rm MSVS1}} (M)$ is superharmonic, which indicates the minimaxity of the generalized Bayes estimator with respect to $\pi_{{\rm MSVS1}}$ in \eqref{MSVS1} under the Frobenius loss from Stein's classical result \citep{Stein74,Matsuda}.
\end{proof}

It is an interesting problem whether the generalized Bayes estimator with respect to $\pi_{{\rm MSVS1}}$ in \eqref{MSVS1} attains admissibility or not.
In addition to Lemma~\ref{lem_brown}, \cite{Brown71} derived the following sufficient condition for admissibility of generalized Bayes estimators, which may be useful here.
While the condition \eqref{Brown_adm} can be verified by using a similar argument to Theorem~\ref{th_inad}, the verification of the uniform boundedness of $\| \nabla \log m_{\pi}(y) \|$ seems difficult.
We leave further investigation for future work.

\begin{lemma}\citep{Brown71}\label{lem_brown2}
	In estimation of $\theta$ from $Y \sim \mathrm{N}_d(\theta,I_d)$ under the quadratic loss, the generalized Bayes estimator of $\theta$ with respect to a prior $\pi(\theta)$ is admissible if $\| \nabla \log m_{\pi}(y) \|$ is uniformly bounded and
	\begin{align}
		\int_c^{\infty} \frac{r^{1-d}}{\widebar{m} (r)} {\rm d} r = \infty, \label{Brown_adm}
	\end{align}
	where 
	\[
	\widebar{m} (r) = \int {m_{\pi} (y)} {\rm d} U_r(y)
	\]
	and $U_r$ is the uniform measure on the sphere of radius $r$ in $\mathbb{R}^d$.
\end{lemma}

\section{Improvement by additional column-wise shrinkage}\label{sec_column}
Here, instead of scalar shrinkage, we consider priors with additional column-wise shrinkage: \begin{align}
	\pi_{{\rm MSVS2}} (M) = \pi_{{\rm SVS}} (M) \prod_{i=1}^p \| M_{\cdot i} \|^{-\gamma_i}, \label{MSVS2}
\end{align}
where $\gamma_i \geq 0$ for every $i$ and $\| M_{\cdot i} \|$ denotes the norm of the $i$-th column vector of $M$.
Let
\begin{align*}
	m_{{\rm MSVS2}} (Y) = \int p(Y \mid M) \pi_{{\rm MSVS2}} (M) {\rm d} M
\end{align*}
be the marginal density of $Y$ under the prior $\pi_{{\rm MSVS2}} (M)$.

\begin{lemma}\label{lem_MSVS2}
	If $0 \leq \gamma_i \leq p$ for every $i$, then $m_{{\rm MSVS2}} (Y) < \infty$ for every $Y$.
\end{lemma}
\begin{proof}
	Similarly to Lemma~\ref{lem_MSVS1}, it suffices to show that $\pi_{{\rm MSVS2}} (M)$ is locally integrable at $M=O$.
	Consider the neighborhood of $M=O$ defined by $A=\{ M \mid \| M \|_{\mathrm{F}} \leq \varepsilon \}$ for $\varepsilon>0$.
	To evaluate the integral on $A$, we use the variable transformation  from $M$ to $(r_1,\dots,r_p,u_1,\dots,u_p)$, where each $r_i \in [0, \infty)$ and $u_i \in \mathbb{R}^n$ with $\| u_i \| = 1$ are defined by $r_i= \| M_{\cdot i} \|$ and $u_i=M_{\cdot i}/r_i$ for the $i$-th column vector $M_{\cdot i}$ of $M$ so that $M_{\cdot i} = r_i u_i$ (polar coordinate).
	Since ${\rm d} M_{\cdot i} = r_i^{n-1} {\rm d} r_i {\rm d} u_i$,
	\begin{align*}
		{\rm d}M = r_1^{n-1} \dots r_p^{n-1} {\rm d} r_1 \dots {\rm d} r_p {\rm d} u_1  \dots {\rm d} u_p.
	\end{align*}
	Also, from $M^{\top} M = D (U^{\top} U) D $ with $D={\rm diag} (r_1,\dots,r_p)$ and $U=(u_1 \dots u_p)$,
	\begin{align*}
		\pi_{{\rm MSVS2}} (M) &= \det (D)^{-(n-p-1)} \det (U^{\top} U)^{-(n-p-1)/2} \prod_{i=1}^p r_i^{-\gamma_i} \\
		&= \prod_{i=1}^p r_i^{-(n-p-1)-\gamma_i} \cdot \det (U^{\top} U)^{-(n-p-1)/2}.
	\end{align*}
	Thus,
	\begin{align}
		&\int_A \pi_{{\rm MSVS2}} (M) {\rm d} M \nonumber \\
		=& \int_{\| r \| \leq \varepsilon} \left( \prod_{i=1}^p r_i^{p-\gamma_i} \right) {\rm d} r_1 \dots {\rm d} r_p \int \det (U^{\top} U)^{-(n-p-1)/2} {\rm d} u_1 \dots {\rm d} u_p. \label{MSVS2int}
	\end{align}
	By variable transformation from $r=(r_1,\dots,r_p)$ to $s=\| r \|$ and $v=r/s$, the first integral in \eqref{MSVS2int} is reduced to
	\begin{align*}
		\int_0^\varepsilon s^{p^2-\sum_{i=1}^p \gamma_i+p-1} {\rm d} s \int_{\| v \| =1} \left( \prod_{i=1}^p v_i^{p-\gamma_i} \right) {\rm d} v.
	\end{align*}
	The integral with respect to $s$ is finite if $p^2-\sum_{i=1}^p \gamma_i+p-1>-1$, which is equivalent to $\sum_{i=1}^p \gamma_i < p^2+p$.
	The integral with respect to $v$ is finite if $p-\gamma_i \geq 0$ for every $i$.
	On the other hand, the second integral in \eqref{MSVS2int} is finite due to the local integrability of $\pi_{\mathrm{SVS}}$, which corresponds to $\gamma=0$, at $M=O$.
	Therefore, $\pi_{{\rm MSVS2}} (M)$ is locally integrable at $M=O$ if $0 \leq \gamma_i \leq p$ for every $i$.
\end{proof}

From Lemma~\ref{lem_MSVS2}, the generalized Bayes estimator with respect to $\pi_{{\rm MSVS2}}$ is well-defined when $0 \leq \gamma \leq p$.
We denote it by $\hat{M}_{{\rm MSVS2}}$.

\begin{theorem}\label{th_MSVS2_dom}
	For every $M$,
	\begin{align}
		N^2 ({\rm E}_M [\| \hat{M}_{{{\rm MSVS2}}}-M \|_{{\rm F}}^2] - {\rm E}_M [\| \hat{M}_{{{\rm SVS}}}-M \|_{{\rm F}}^2]) \to \sum_{i=1}^p \gamma_i (\gamma_i-2p+2) \| M_{\cdot i} \|^{-2} \label{risk_diff2}
	\end{align}
	as $N \to \infty$.
	Therefore, if $p \geq 2$ and $0 < \gamma_i \leq p$ for every $i$, then the generalized Bayes estimator with respect to $\pi_{{\rm MSVS2}}$ in \eqref{MSVS2} asymptotically dominates that with respect to $\pi_{{\rm SVS}}$ in \eqref{SVS} under the Frobenius loss.
\end{theorem}
\begin{proof}
	Let
	\begin{align*}
		\pi_{{\rm CS}} (M) = \prod_{i=1}^p \| M_{\cdot i} \|^{-\gamma_i}.
	\end{align*}
	Then,
	\begin{align}
		\frac{\partial}{\partial M_{ai}} \log \pi_{{\rm CS}} (M) = -\gamma_i M_{ai} \| M_{\cdot i} \|^{-2}, \label{nabla4}
	\end{align}
	\begin{align}
		\frac{\partial^2}{\partial M_{ai}^2} \log \pi_{{\rm CS}} (M) = -\gamma_i \left( \| M_{\cdot i} \|^2 - 2 M_{ai}^2 \right) \| M_{\cdot i} \|^{-4}. \label{nabla5}
	\end{align}	
	
	From \eqref{nabla1}, \eqref{nabla4}, and \eqref{nabla5},
	\begin{align}
		{\rm tr} (\widetilde{\nabla} \log \pi_{{\rm SVS}} (M)^{\top} \widetilde{\nabla} \log \pi_{{\rm CS}} (M) ) = (n-p-1) \sum_{i=1}^p \gamma_i \| M_{\cdot i} \|^{-2}, \label{nabla6} \\
		{\rm tr} (\widetilde{\nabla} \log \pi_{{\rm CS}} (M)^{\top} \widetilde{\nabla} \log \pi_{{\rm CS}} (M) ) = \sum_{i=1}^p \gamma_i^2 \| M_{\cdot i} \|^{-2}, \label{nabla7} \\
		{\rm tr} (\widetilde{\Delta} \log \pi_{{\rm CS}} (M)) = -(n-2) \sum_{i=1}^p \gamma_i \| M_{\cdot i} \|^{-2}, \label{nabla8}
	\end{align}
	where we used the matrix derivative notations \eqref{matrix_grad} and \eqref{matrix_lap}.
	Therefore, from Lemma~\ref{lem_mat_est},
	\begin{align*}
		& {\rm E}_M [\| \hat{M}_{{\rm MSVS2}}-M \|_{{\rm F}}^2] - {\rm E}_M [\| \hat{M}_{{\rm SVS}}-M \|_{{\rm F}}^2] \\
		= & \frac{1}{N^2} {\rm tr}  ( 2 \widetilde{\nabla} \log \pi_{{\rm SVS}} (M)^{\top} \widetilde{\nabla} \log \pi_{{\rm CS}} (M) + \widetilde{\nabla} \log \pi_{{\rm CS}} (M)^{\top} \widetilde{\nabla} \log \pi_{{\rm CS}} (M) + 2 \widetilde{\Delta} \log \pi_{{\rm CS}} (M) ) \\
		& \quad + o(N^{-2}) \\
		= & \frac{1}{N^2} \sum_{i=1}^p \gamma_i (\gamma_i-2p+2)  \| M_{\cdot i} \|^{-2} + o(N^{-2}). 	
	\end{align*}
Hence, we obtain \eqref{risk_diff2}.
\end{proof}

From \eqref{risk_diff2}, the choice $\gamma_1=\dots=\gamma_p = p-1$ attains the minimum risk among $0 < \gamma_i \leq p$.

Now, we examine the performance of $\pi_{{\rm MSVS2}}$ in \eqref{MSVS2}  by Monte Carlo simulation.
Figure~\ref{fig_MSVS2} plots the Frobenius risk of generalized Bayes estimators with respect to $\pi_{{\rm MSVS2}}$ in \eqref{MSVS2} with $\gamma_1=\dots=\gamma_p=p-1$, $\pi_{{\rm MSVS1}}$ in \eqref{MSVS1} with $\gamma=p^2+p-2$, $\pi_{{\rm SVS}}$ in \eqref{SVS} and $\pi_{{\rm S}}(M)=\| M \|_{{\rm F}}^{2-np}$, which is Stein's prior on the vectorization of $M$.
We computed the generalized Bayes estimators by using the random walk Metropolis--Hastings algorithm with proposal variance $0.1$.
We set $M=U \Sigma$, where $U^{\top} U=I_p$ and $\Sigma={\rm diag}(\sigma_1,\dots,\sigma_p)$. 
For this $M$, the Frobenius risk of the estimators compared here depends only on the singular values $\sigma_1,\dots,\sigma_p$ of $M$.
Overall, $\pi_{{\rm MSVS2}}$ performs better than $\pi_{{\rm SVS}}$.
Also, $\pi_{{\rm MSVS2}}$ even dominates $\pi_{{\rm MSVS1}}$ and $\pi_{{\rm S}}$ except when $\sigma_1$ is sufficiently small.
This is understood from the column-wise shrinkage effect of $\pi_{{\rm MSVS2}}$.
Figure~\ref{fig_MSVS2_N10} plots the Frobenius risk for $n=10$, $p=3$ and $N=10$, computed by the random walk Metropolis--Hastings algorithm with proposal variance $0.01$.
The risk behavior is similar to Figure~\ref{fig_MSVS2}.
Figure~\ref{fig_MSVS2_highdim} plots the Frobenius risk for $n=20$, $p=3$ and $N=2$, computed by the random walk Metropolis--Hastings algorithm with proposal variance $0.01$.
Again, the risk behavior is similar to Figure~\ref{fig_MSVS2}.
Note that the value of $np/N=30$ is the same with Figure~\ref{fig_MSVS2}.

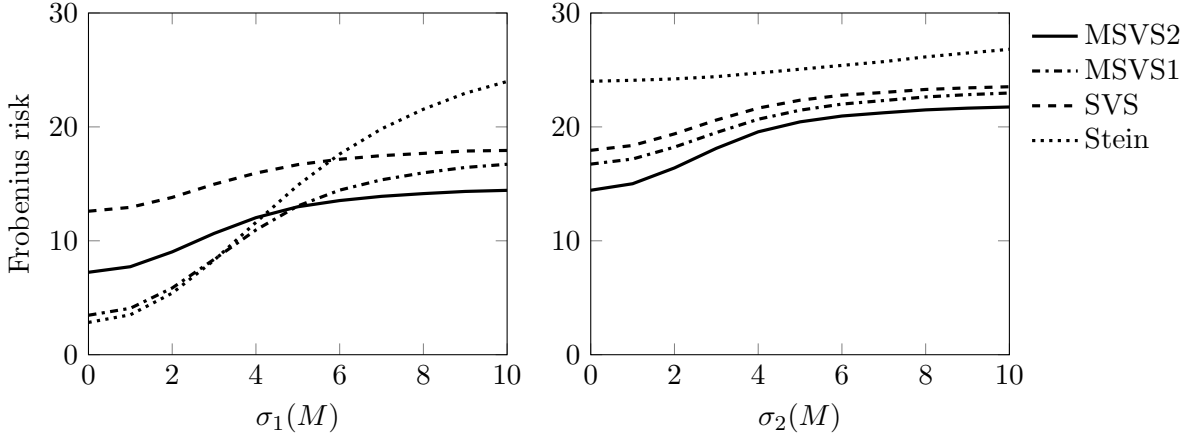
\begin{figure}[htbp]
	\centering
	\begin{tikzpicture}
\tikzstyle{every node}=[]
		\begin{axis}[
			width=0.45\linewidth,
			xmax=10,xmin=0,
			ymax=30, ymin = 0,
xlabel={$\sigma_1(M)$},ylabel={Frobenius risk},
ylabel near ticks,
			]
\addplot[very thick,
			filter discard warning=false, unbounded coords=discard
			] table {
         0    7.2376
1.0000    7.7263
2.0000    9.0292
3.0000   10.6485
4.0000   12.0315
5.0000   12.9888
6.0000   13.5339
7.0000   13.9037
8.0000   14.1437
9.0000   14.3344
10.0000   14.4301
			};
			\addplot[very thick, dash dot,
			filter discard warning=false, unbounded coords=discard
			] table {
         0    3.4656
1.0000    4.0759
2.0000    5.8565
3.0000    8.3702
4.0000   10.9523
5.0000   13.0605
6.0000   14.4480
7.0000   15.3545
8.0000   15.9686
9.0000   16.4397
10.0000   16.7205
			};
			\addplot[very thick, dashed,
			filter discard warning=false, unbounded coords=discard
			] table {
         0   12.5978
1.0000   12.9375
2.0000   13.8103
3.0000   14.9662
4.0000   15.9446
5.0000   16.6890
6.0000   17.1596
7.0000   17.4767
8.0000   17.6752
9.0000   17.8829
10.0000   17.9318
			};
			\addplot[very thick, dotted,
			filter discard warning=false, unbounded coords=discard
			] table {
         0    2.8370
1.0000    3.5037
2.0000    5.4061
3.0000    8.2906
4.0000   11.6186
5.0000   14.8771
6.0000   17.6187
7.0000   19.8293
8.0000   21.5501
9.0000   22.9553
10.0000   23.9878
			};
		\end{axis}
	\end{tikzpicture} 
	\begin{tikzpicture}
\tikzstyle{every node}=[]
		\begin{axis}[
			width=0.45\linewidth,
			xmax=10,xmin=0,
			ymax=30, ymin = 0,
xlabel={$\sigma_2(M)$},ylabel near ticks,
			legend pos=outer north east,
			legend entries={MSVS2,MSVS1,SVS,Stein},
			legend style={legend cell align=left,draw=none,fill=white,fill opacity=0.8,text opacity=1,},
			]
\addplot[very thick,
			filter discard warning=false, unbounded coords=discard
			] table {
         0   14.4381
1.0000   15.0104
2.0000   16.3979
3.0000   18.1193
4.0000   19.5685
5.0000   20.4432
6.0000   20.9555
7.0000   21.2402
8.0000   21.4952
9.0000   21.6454
10.0000   21.7492
			};
			\addplot[very thick, dash dot,
			filter discard warning=false, unbounded coords=discard
			] table {
         0   16.7335
1.0000   17.1809
2.0000   18.2266
3.0000   19.5041
4.0000   20.6708
5.0000   21.4675
6.0000   22.0029
7.0000   22.3146
8.0000   22.6281
9.0000   22.8322
10.0000   22.9755
			};
			\addplot[very thick, dashed,
			filter discard warning=false, unbounded coords=discard
			] table {
         0   17.9455
1.0000   18.3686
2.0000   19.3874
3.0000   20.6019
4.0000   21.6507
5.0000   22.3445
6.0000   22.7839
7.0000   23.0275
8.0000   23.2875
9.0000   23.4274
10.0000   23.5259
			};
			\addplot[very thick, dotted,
			filter discard warning=false, unbounded coords=discard
			] table {
         0   24.0056
1.0000   24.0899
2.0000   24.2154
3.0000   24.4113
4.0000   24.7333
5.0000   25.0638
6.0000   25.3991
7.0000   25.7304
8.0000   26.1471
9.0000   26.4873
10.0000   26.8118
			};
		\end{axis}
	\end{tikzpicture} 
	\caption{Frobenius risk of generalized Bayes estimators for $n=10$, $p=3$ and $N=1$ where $M=U \Sigma$ with $U^{\top} U=I_p$ and $\Sigma={\rm diag}(\sigma_1,\dots,\sigma_p)$. Left: $\sigma_2=\sigma_3=0$. Right: $\sigma_1=10$, $\sigma_3=0$. solid: $\pi_{{\rm MSVS2}}$ with $\gamma_1=\dots=\gamma_p=p-1$, dash-dotted: $\pi_{{\rm MSVS1}}$ with $\gamma=p^2+p-2$, dashed: $\pi_{{\rm SVS}}$, dotted: Stein's prior $\pi_{{\rm S}}$. Note that the minimax risk is $np/N=30$.}
	\label{fig_MSVS2}
\end{figure}
 
\begin{figure}[htbp]
	\centering
	\begin{tikzpicture}
\tikzstyle{every node}=[]
		\begin{axis}[
			width=0.45\linewidth,
			xmax=10,xmin=0,
			ymax=3, ymin = 0,
xlabel={$\sigma_1(M)$},ylabel={Frobenius risk},
ylabel near ticks,
			]
\addplot[very thick,
			filter discard warning=false, unbounded coords=discard
			] table {
         0    0.6307
1.0000    1.0180
2.0000    1.2884
3.0000    1.3642
4.0000    1.3893
5.0000    1.4073
6.0000    1.4058
7.0000    1.4169
8.0000    1.4081
9.0000    1.4283
10.0000    1.4265
			};
			\addplot[very thick, dash dot,
			filter discard warning=false, unbounded coords=discard
			] table {
         0    0.2534
1.0000    0.8315
2.0000    1.4369
3.0000    1.6199
4.0000    1.6911
5.0000    1.7354
6.0000    1.7410
7.0000    1.7667
8.0000    1.7657
9.0000    1.7904
10.0000    1.7898
			};
			\addplot[very thick, dashed,
			filter discard warning=false, unbounded coords=discard
			] table {
         0    1.1999
1.0000    1.4689
2.0000    1.6743
3.0000    1.7484
4.0000    1.7691
5.0000    1.7890
6.0000    1.7794
7.0000    1.7947
8.0000    1.7871
9.0000    1.8074
10.0000    1.8052
			};
			\addplot[very thick, dotted,
			filter discard warning=false, unbounded coords=discard
			] table {
         0    0.2285
1.0000    0.8571
2.0000    1.8176
3.0000    2.3278
4.0000    2.5758
5.0000    2.7206
6.0000    2.7815
7.0000    2.8488
8.0000    2.8704
9.0000    2.9184
10.0000    2.9270
			};
		\end{axis}
	\end{tikzpicture} 
	\begin{tikzpicture}
\tikzstyle{every node}=[]
		\begin{axis}[
			width=0.45\linewidth,
			xmax=10,xmin=0,
			ymax=3, ymin = 0,
xlabel={$\sigma_2(M)$},ylabel near ticks,
			legend pos=outer north east,
			legend entries={MSVS2,MSVS1,SVS,Stein},
			legend style={legend cell align=left,draw=none,fill=white,fill opacity=0.8,text opacity=1,},
			]
\addplot[very thick,
			filter discard warning=false, unbounded coords=discard
			] table {
         0    1.4187
1.0000    1.8449
2.0000    2.0991
3.0000    2.1605
4.0000    2.1795
5.0000    2.1927
6.0000    2.1918
7.0000    2.2052
8.0000    2.1994
9.0000    2.2021
10.0000    2.2005
			};
			\addplot[very thick, dash dot,
			filter discard warning=false, unbounded coords=discard
			] table {
         0    1.7884
1.0000    2.0872
2.0000    2.2822
3.0000    2.3448
4.0000    2.3578
5.0000    2.3739
6.0000    2.3727
7.0000    2.3873
8.0000    2.3781
9.0000    2.3839
10.0000    2.3850
			};
			\addplot[very thick, dashed,
			filter discard warning=false, unbounded coords=discard
			] table {
         0    1.7999
1.0000    2.1008
2.0000    2.2950
3.0000    2.3550
4.0000    2.3700
5.0000    2.3820
6.0000    2.3812
7.0000    2.3962
8.0000    2.3854
9.0000    2.3897
10.0000    2.3907
			};
			\addplot[very thick, dotted,
			filter discard warning=false, unbounded coords=discard
			] table {
         0    2.9279
1.0000    2.9357
2.0000    2.9228
3.0000    2.9367
4.0000    2.9279
5.0000    2.9421
6.0000    2.9357
7.0000    2.9575
8.0000    2.9458
9.0000    2.9533
10.0000    2.9581
			};
		\end{axis}
	\end{tikzpicture} 
	\caption{Frobenius risk of generalized Bayes estimators for $n=10$, $p=3$ and $N=10$ where $M=U \Sigma$ with $U^{\top} U=I_p$ and $\Sigma={\rm diag}(\sigma_1,\dots,\sigma_p)$. Left: $\sigma_2=\sigma_3=0$. Right: $\sigma_1=10$, $\sigma_3=0$. solid: $\pi_{{\rm MSVS2}}$ with $\gamma_1=\dots=\gamma_p=p-1$, dash-dotted: $\pi_{{\rm MSVS1}}$ with $\gamma=p^2+p-2$, dashed: $\pi_{{\rm SVS}}$, dotted: Stein's prior $\pi_{{\rm S}}$. Note that the minimax risk is $np/N=3$.}
	\label{fig_MSVS2_N10}
\end{figure}
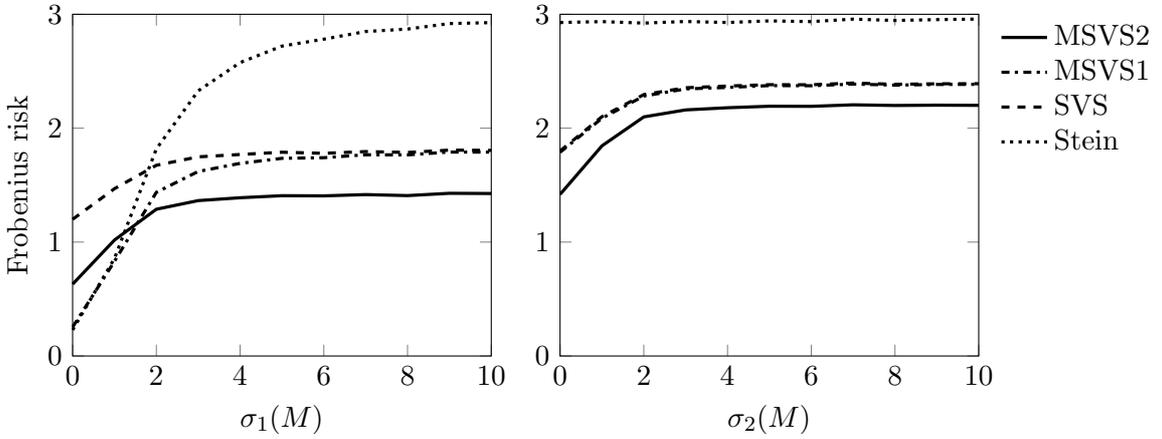
 
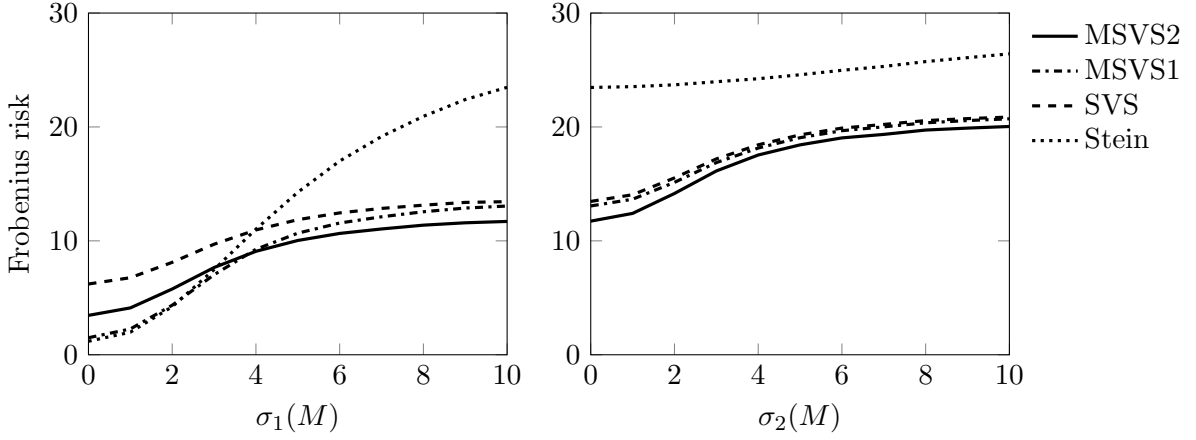
\begin{figure}[htbp]
	\centering
	\begin{tikzpicture}
\tikzstyle{every node}=[]
		\begin{axis}[
			width=0.45\linewidth,
			xmax=10,xmin=0,
			ymax=30, ymin = 0,
xlabel={$\sigma_1(M)$},ylabel={Frobenius risk},
ylabel near ticks,
			]
\addplot[very thick,
			filter discard warning=false, unbounded coords=discard
			] table {
         0    3.4580
1.0000    4.1146
2.0000    5.7725
3.0000    7.6584
4.0000    9.0821
5.0000   10.0346
6.0000   10.6525
7.0000   11.0442
8.0000   11.3785
9.0000   11.5863
10.0000   11.7054
			};
			\addplot[very thick, dash dot,
			filter discard warning=false, unbounded coords=discard
			] table {
         0    1.4953
1.0000    2.2639
2.0000    4.3492
3.0000    7.0251
4.0000    9.2623
5.0000   10.6756
6.0000   11.5667
7.0000   12.1130
8.0000   12.5553
9.0000   12.8812
10.0000   13.0567
			};
			\addplot[very thick, dashed,
			filter discard warning=false, unbounded coords=discard
			] table {
         0    6.2128
1.0000    6.7650
2.0000    8.1142
3.0000    9.7041
4.0000   10.9588
5.0000   11.8599
6.0000   12.4554
7.0000   12.8489
8.0000   13.1330
9.0000   13.3760
10.0000   13.4415
			};
			\addplot[very thick, dotted,
			filter discard warning=false, unbounded coords=discard
			] table {
         0    1.1859
1.0000    1.9794
2.0000    4.2208
3.0000    7.4770
4.0000   11.0274
5.0000   14.2669
6.0000   17.0075
7.0000   19.1622
8.0000   20.9322
9.0000   22.4048
10.0000   23.4788
			};
		\end{axis}
	\end{tikzpicture} 
	\begin{tikzpicture}
\tikzstyle{every node}=[]
		\begin{axis}[
			width=0.45\linewidth,
			xmax=10,xmin=0,
			ymax=30, ymin = 0,
xlabel={$\sigma_2(M)$},ylabel near ticks,
			legend pos=outer north east,
			legend entries={MSVS2,MSVS1,SVS,Stein},
			legend style={legend cell align=left,draw=none,fill=white,fill opacity=0.8,text opacity=1,},
			]
\addplot[very thick,
			filter discard warning=false, unbounded coords=discard
			] table {
         0   11.7315
1.0000   12.4105
2.0000   14.1606
3.0000   16.1359
4.0000   17.5363
5.0000   18.4278
6.0000   19.0392
7.0000   19.3466
8.0000   19.7282
9.0000   19.8987
10.0000   20.0408
			};
			\addplot[very thick, dash dot,
			filter discard warning=false, unbounded coords=discard
			] table {
         0   13.0643
1.0000   13.6548
2.0000   15.1348
3.0000   16.8697
4.0000   18.1485
5.0000   19.0569
6.0000   19.6719
7.0000   20.0096
8.0000   20.3659
9.0000   20.5606
10.0000   20.7171
			};
			\addplot[very thick, dashed,
			filter discard warning=false, unbounded coords=discard
			] table {
         0   13.4639
1.0000   14.0499
2.0000   15.5257
3.0000   17.1941
4.0000   18.4311
5.0000   19.3059
6.0000   19.8997
7.0000   20.2207
8.0000   20.5587
9.0000   20.7324
10.0000   20.8678
			};
			\addplot[very thick, dotted,
			filter discard warning=false, unbounded coords=discard
			] table {
         0   23.4668
1.0000   23.5384
2.0000   23.6988
3.0000   23.9707
4.0000   24.2344
5.0000   24.5813
6.0000   24.9738
7.0000   25.3125
8.0000   25.7376
9.0000   26.0838
10.0000   26.4161
			};
		\end{axis}
	\end{tikzpicture} 
	\caption{Frobenius risk of generalized Bayes estimators for $n=20$, $p=3$ and $N=2$ where $M=U \Sigma$ with $U^{\top} U=I_p$ and $\Sigma={\rm diag}(\sigma_1,\dots,\sigma_p)$. Left: $\sigma_2=\sigma_3=0$. Right: $\sigma_1=10$, $\sigma_3=0$. solid: $\pi_{{\rm MSVS2}}$ with $\gamma_1=\dots=\gamma_p=p-1$, dash-dotted: $\pi_{{\rm MSVS1}}$ with $\gamma=p^2+p-2$, dashed: $\pi_{{\rm SVS}}$, dotted: Stein's prior $\pi_{{\rm S}}$. Note that the minimax risk is $np/N=30$.}
	\label{fig_MSVS2_highdim}
\end{figure}
 
Improvement by additional column-wise shrinkage holds even under the matrix quadratic loss \citep{Matsuda22,Matsuda23b}.

\begin{theorem}
	For every $M$,
	\begin{align}
		N^2 ({\rm E}_M [ (\hat{M}_{{{\rm MSVS2}}}-M)^{\top} &(\hat{M}_{{{\rm MSVS2}}}-M) ]- {\rm E}_M [ (\hat{M}_{{{\rm SVS}}}-M)^{\top} (\hat{M}_{{{\rm SVS}}}-M) ]) \nonumber \\
		&\to -2(p-1) D + D M^{\top} M D \label{risk_diff_mat3}
	\end{align}
	as $N \to \infty$, where $D={\rm diag}(\gamma_1 \| M_{\cdot 1} \|^{-2},\dots,\gamma_p \| M_{\cdot p} \|^{-2})$.
	Therefore, if $p \geq 2$ and $0 < \gamma_1=\dots=\gamma_p < 2-2/p$, then the generalized Bayes estimator with respect to $\pi_{{\rm MSVS2}}$ in \eqref{MSVS2} asymptotically dominates that with respect to $\pi_{{\rm SVS}}$ in \eqref{SVS} under the matrix quadratic loss.
\end{theorem}
\begin{proof}
	We use the same notation with the proof of Theorem~\ref{th_MSVS2_dom}.
	By using \eqref{nabla1}, \eqref{nabla4}, and \eqref{nabla5}, we obtain
	\begin{align*}
		\widetilde{\nabla} \log \pi_{{\rm SVS}} (M)^{\top} \widetilde{\nabla} \log \pi_{{\rm CS}} (M) = (n-p-1) D, \\
		\widetilde{\nabla} \log \pi_{{\rm CS}} (M)^{\top} \widetilde{\nabla} \log \pi_{{\rm CS}} (M) = D M^{\top} M D, \\
		\widetilde{\Delta} \log \pi_{{\rm CS}} (M) = -(n-2) D.
	\end{align*}
	Therefore, from Lemma~\ref{lem_mat_est2},
	\begin{align*}
		& {\rm E}_M [ (\hat{M}_{{{\rm MSVS2}}}-M)^{\top} (\hat{M}_{{{\rm MSVS2}}}-M) ] - {\rm E}_M [ (\hat{M}_{{{\rm SVS}}}-M)^{\top} (\hat{M}_{{{\rm SVS}}}-M) ] \\
		= & \frac{1}{N^2} ( 2 \widetilde{\nabla} \log \pi_{{\rm SVS}} (M)^{\top} \widetilde{\nabla} \log \pi_{{\rm CS}} (M) + \widetilde{\nabla} \log \pi_{{\rm CS}} (M)^{\top} \widetilde{\nabla} \log \pi_{{\rm CS}} (M) + 2 \widetilde{\Delta} \log \pi_{{\rm CS}} (M) ) \\
		& \quad + o(N^{-2}) \\
		= & -2(p-1) D + D M^{\top} M D.
	\end{align*}
	Hence, we obtain \eqref{risk_diff_mat3}.
	
	Suppose that $p \geq 2$ and $0 < \gamma_1=\dots=\gamma_p < 2-2/p$.
	Let $\| \cdot \|$ be the operater norm.
	Since $D \succeq O$ is diagonal and $\| M \|^2 \leq \| M \|_{\mathrm{F}}^2=\sum_{i=1}^p \| M_{\cdot i} \|^2$,
	\begin{align*}
		\| D^{1/2} M^{\top} M D^{1/2} \| & \leq \| D^{1/2} \| \| M^{\top} M \| \| D^{1/2} \| \\
		&= ( \max_i D_{ii} ) \| M \|^2 \\
		&= \gamma_1 \frac{\| M \|^2}{\min_i \| M_{\cdot i} \|^2} \\
		&\leq \gamma_1 \frac{\sum_{i=1}^p \| M_{\cdot i} \|^2}{\min_i \| M_{\cdot i} \|^2} \\
		&\leq \gamma_1 p \\
		&< 2(p-1),
	\end{align*}
	which yields $D^{1/2} M^{\top} M D^{1/2} \prec 2(p-1) I_p$.
	Therefore, 
	\begin{align*}
		-2(p-1) D + D M^{\top} M D = D^{1/2} (-2(p-1) I_p + D^{1/2} M^{\top} M D^{1/2}) D^{1/2} \prec O.
	\end{align*}
\end{proof}

The generalized Bayes estimator with respect to $\pi_{{\rm MSVS2}}$ in \eqref{MSVS2} attains minimaxity in some cases as follows.
It is an interesting future work to investigate its admissibility.

\begin{theorem}
	If $p \geq 3$, $p+2 \leq n < 2p$ and $0 < \gamma \leq -n+2p$, then the generalized Bayes estimator with respect to $\pi_{{\rm MSVS2}}$ in \eqref{MSVS2} is minimax under the Frobenius loss.
\end{theorem}
\begin{proof}
	From Proposition~\ref{prop_MSVS2_lap}, 
	\begin{align*}
		\Delta \pi_{{\rm MSVS2}} (M) &= \gamma (\gamma+n-2p) \left( \sum_{i=1}^p \| M_{\cdot i} \|^{-2} \right) \pi_{{\rm MSVS2}} (M) \leq 0.
	\end{align*}
	Thus, $\pi_{{\rm MSVS2}} (M)$ is superharmonic, which indicates the minimaxity of the generalized Bayes estimator with respect to $\pi_{{\rm MSVS2}}$ in \eqref{MSVS2} under the Frobenius loss from Stein's classical result \citep{Stein74,Matsuda}.
\end{proof}

\section{Bayesian prediction}\label{sec_pred}
Here, we consider Bayesian prediction and provide parallel results to those in Sections~\ref{sec_scalar} and \ref{sec_column}.
Suppose that we observe $Y \sim {\rm N}_{n,p} (M, I_n, N^{-1} I_p)$ and predict $\widetilde{Y} \sim {\rm N}_{n,p} (M, I_n, I_p)$ by a predictive density $\hat{p}(\widetilde{Y} \mid Y)$. 
We evaluate predictive densities by the Kullback--Leibler loss:
\begin{align*}
	D ({p}(\cdot \mid M), \hat{p} (\cdot \mid Y))
	= \int {p}(\widetilde{Y} \mid M) \log \frac{{p}(\widetilde{Y} \mid M)}{\hat{p} (\widetilde{Y} \mid Y)} {\rm d} \widetilde{Y}. 
\end{align*}
The Bayesian predictive density based on a prior $\pi(M)$ is defined as
\begin{align*}
	\hat{p}_{\pi} (\widetilde{Y} \mid Y) & = \int {p} (\widetilde{Y} \mid M) \pi (M \mid Y) {\rm d} M,
\end{align*}
where $\pi(M \mid Y)$ is the posterior distribution of $M$ given $Y$, and it minimizes the Bayes risk \citep{Aitchison}:
\begin{align*}
	\hat{p}_{\pi} (\widetilde{Y} \mid Y) & = \argmin_{\hat{p}} \int D ({p}(\cdot \mid M), \hat{p} (\cdot \mid Y)) p(Y \mid M) \pi (M) {\rm d} Y {\rm d} M.
\end{align*}
The Bayesian predictive density with respect to the uniform prior is minimax.
However, it is inadmissible and dominated by Bayesian predictive densities based on superharmonic priors \citep{Komaki01,George06}.
In particular, the Bayesian predictive density based on the singular value shrinkage prior $\pi_{{\rm SVS}}$ in \eqref{SVS} is minimax and dominates that based on the uniform prior \citep{Matsuda}.

The asymptotic expansion of the difference between the Kullback--Leibler risk of two Bayesian predictive densities is obtained as follows.

\begin{lemma}\label{lem_mat_pred1}
	As $N \to \infty$, the difference between the Kullback--Leibler risk of $p_{\pi_1} (\widetilde{Y} \mid Y)$ and $p_{\pi_1 \pi_2} (\widetilde{Y} \mid Y)$ is expanded as
	\begin{align}
		&  {\rm E}_{M} [ D(p(\widetilde{Y} \mid M), p_{\pi_1 \pi_2}(\widetilde{Y} \mid Y)) ] - {\rm E}_{M} [ D(p(\widetilde{Y} \mid M), p_{\pi_1}(\widetilde{Y} \mid Y)) ] \nonumber \\
		= & \frac{1}{2 N^2} {\rm tr} ( 2 (\widetilde{\nabla} \log \pi_1 (M))^{\top} (\widetilde{\nabla} \log \pi_2 (M)) + (\widetilde{\nabla} \log \pi_2 (M))^{\top} (\widetilde{\nabla} \log \pi_2 (M)) + 2 \widetilde{\Delta} \log \pi_2 (M) ) \nonumber \\
		& \quad + o(N^{-2}). \label{pred_risk_diff}
	\end{align}
\end{lemma}
\begin{proof}
	For the normal model with known covariance, the information geometrical quantities \citep{Amari} are given by
	\begin{align*}
		g_{i j} = g^{i j} = \delta_{i j}, \quad \Gamma_{i j}^k = 0, \quad T_{i j k} = 0.
	\end{align*}
	Also, the Jeffreys prior coincides with the uniform prior $\pi (M) \equiv 1$.
	Therefore, from equation (3) of \cite{Komaki06}, the Kullback--Leibler risk of the Bayesian predictive density 
	$p_{\pi}(\widetilde{Y} \mid Y)$ based on a prior $\pi(M)$ is expanded as
	\begin{align}
		& {\rm E}_{M} [ D(p(\widetilde{Y} \mid M), p_{\pi}(\widetilde{Y} \mid Y)) ] \nonumber \\
		= & \frac{np}{2 N} +\frac{1}{2 N^2} {\rm tr} ((\widetilde{\nabla} \log \pi(M))^{\top} (\widetilde{\nabla} \log \pi(M)) + 2 \widetilde{\Delta} \log \pi(M)) + g(M) + o(N^{-2}), \label{pred_asymp_risk}
	\end{align}
	where $g(M)$ is a function independent of $\pi(M)$.
	Substituting $\pi=\pi_1 \pi_2$ and $\pi=\pi_1$ into \eqref{pred_asymp_risk} and taking difference, we obtain \eqref{pred_risk_diff}.
\end{proof}

By comparing Lemma~\ref{lem_mat_pred1} to Lemma~\ref{lem_mat_est}, we obtain the following connection between estimation and prediction.

\begin{proposition}\label{lem_equiv}
	For every $M$,
	\begin{align*}
		&\lim_{N \to \infty} N^2 ( {\rm E}_{M} [ D(p(\widetilde{Y} \mid M), \hat{p}_{\pi_1 \pi_2}(\widetilde{Y} \mid Y)) ] - {\rm E}_{M} [ D(p(\widetilde{Y} \mid M), \hat{p}_{\pi_1}(\widetilde{Y} \mid Y)) ] ) \\
		=& \frac{1}{2} \lim_{N \to \infty} N^2 ( {\rm E}_{M} [\| \hat{M}^{\pi_1 \pi_2}-M \|_{\mathrm{F}}^2] - {\rm E}_{M} [\| \hat{M}^{\pi_1}-M \|_{\mathrm{F}}^2] ).
	\end{align*}
	Therefore, if $\hat{\theta}^{\pi_1 \pi_2}$ asymptotically dominates $\hat{\theta}^{\pi_1}$ under the quadratic loss, then $\hat{p}_{\pi_1 \pi_2}(\widetilde{Y} \mid Y)$ asymptotically dominates $\hat{p}_{\pi_1}(\widetilde{Y} \mid Y)$ under the Kullback--Leibler loss.
\end{proposition}

Therefore, Theorems~\ref{th_MSVS1_dom} and \ref{th_MSVS2_dom} are extended to Bayesian prediction as follows.
Other results in the previous sections can be extended to Bayesian prediction similarly.

\begin{theorem}\label{th_MSVS1_pred}
	For every $M$,
	\begin{align*}
		& N^2 ({\rm E}_M [ D(p(\cdot \mid M),\hat{p}_{{\rm MSVS1}}(\cdot \mid Y) ) ] - {\rm E}_M [ D(p(\cdot \mid M),\hat{p}_{{\rm SVS}}(\cdot \mid Y) ) ] ) \\
		\to& \frac{\gamma (\gamma-2p^2-2p+4)}{2 {\rm tr} (M^{\top} M)}
	\end{align*}
	as $N \to \infty$.
	Therefore, if $p \geq 2$ and $0 < \gamma < p^2+p$, then the Bayesian predictive density with respect to $\pi_{{\rm MSVS1}}$ in \eqref{MSVS1} asymptotically dominates that with respect to $\pi_{{\rm SVS}}$ in \eqref{SVS} under the Kullback--Leibler loss.
\end{theorem}

\begin{theorem}\label{th_MSVS2_pred}
	For every $M$,
	\begin{align*}
		& N^2 ({\rm E}_M [ D(p(\cdot \mid M),\hat{p}_{{\rm MSVS2}}(\cdot \mid Y) ) ] - {\rm E}_M [ D(p(\cdot \mid M),\hat{p}_{{\rm SVS}}(\cdot \mid Y) ) ] ) \\
		\to& \frac{1}{2} \sum_{i=1}^p \gamma_i (\gamma_i-2p+2) \| M_{\cdot i} \|^{-2}
	\end{align*}
	as $N \to \infty$.
	Therefore, if $p \geq 2$ and $0 < \gamma \leq p$, then the Bayesian predictive density with respect to $\pi_{{\rm MSVS2}}$ in \eqref{MSVS2} asymptotically dominates that with respect to $\pi_{{\rm SVS}}$ in \eqref{SVS} under the Kullback--Leibler loss.
\end{theorem}

Figures~\ref{fig_MSVS1_pred} and \ref{fig_MSVS2_pred} plot the Kullback--Leibler risk of Bayesian predictive densities in similar settings to Figures~\ref{fig_MSVS1} and \ref{fig_MSVS2}, respectively.
They show that the risk behavior in prediction is qualitatively the same with that in estimation, which is compatible with Theorems~\ref{th_MSVS1_pred} and \ref{th_MSVS2_pred}.

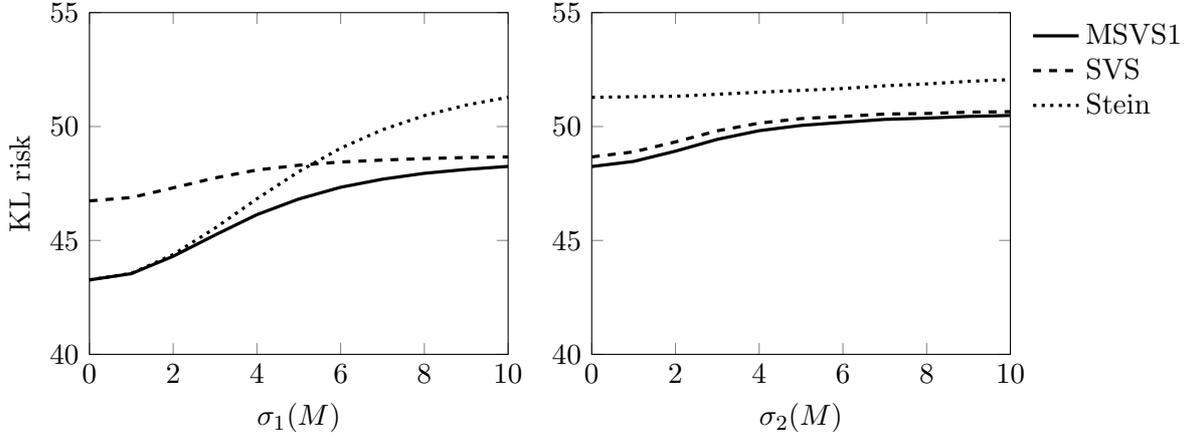
\begin{figure}[htbp]
	\centering
\begin{tikzpicture}
\tikzstyle{every node}=[]
	\begin{axis}[
		width=0.45\linewidth,
		xmax=10,xmin=0,
		ymax=55, ymin = 40,
xlabel={$\sigma_1(M)$},ylabel={KL risk},
ylabel near ticks,
]
\addplot[very thick,
		filter discard warning=false, unbounded coords=discard
		] table {
         0   43.2673
1.0000   43.5413
2.0000   44.3043
3.0000   45.2369
4.0000   46.1322
5.0000   46.8175
6.0000   47.3343
7.0000   47.6883
8.0000   47.9465
9.0000   48.1175
10.0000   48.2497
		};
		\addplot[very thick, dashed,
		filter discard warning=false, unbounded coords=discard
		] table {
         0   46.7354
1.0000   46.8835
2.0000   47.3083
3.0000   47.7431
4.0000   48.0924
5.0000   48.3012
6.0000   48.4424
7.0000   48.5275
8.0000   48.5927
9.0000   48.6398
10.0000   48.6652
		};
		\addplot[very thick, dotted,
		filter discard warning=false, unbounded coords=discard
		] table {
         0   43.2675
1.0000   43.5546
2.0000   44.3793
3.0000   45.5416
4.0000   46.8282
5.0000   48.0213
6.0000   49.0499
7.0000   49.8597
8.0000   50.4799
9.0000   50.9415
10.0000   51.2860
		};
	\end{axis}
\end{tikzpicture} 
\begin{tikzpicture}
\begin{axis}[
	width=0.45\linewidth,
	xmax=10,xmin=0,
	ymax=55, ymin = 40,
xlabel={$\sigma_2(M)$},
ylabel near ticks,
	legend pos=outer north east,
	legend entries={MSVS1,SVS,Stein},
	legend style={legend cell align=left,draw=none,fill=white,fill opacity=0.8,text opacity=1,},
	]
\addplot[very thick, 
	filter discard warning=false, unbounded coords=discard
	] table {
         0   48.2434
1.0000   48.4682
2.0000   48.9182
3.0000   49.4368
4.0000   49.8171
5.0000   50.0477
6.0000   50.1834
7.0000   50.3135
8.0000   50.3681
9.0000   50.4446
10.0000   50.4869
	};
	\addplot[very thick, dashed,
	filter discard warning=false, unbounded coords=discard
	] table {
         0   48.6659
1.0000   48.8872
2.0000   49.3236
3.0000   49.8108
4.0000   50.1478
5.0000   50.3484
6.0000   50.4432
7.0000   50.5478
8.0000   50.5723
9.0000   50.6297
10.0000   50.6522
	};
	\addplot[very thick, dotted,
	filter discard warning=false, unbounded coords=discard
	] table {
         0   51.2798
1.0000   51.3098
2.0000   51.3259
3.0000   51.4162
4.0000   51.5062
5.0000   51.5871
6.0000   51.6690
7.0000   51.7891
8.0000   51.8707
9.0000   51.9839
10.0000   52.0612
	};
\end{axis}
\end{tikzpicture} 
	\caption{Kullback--Leibler risk of Bayesian predictive densities for $n=10$, $p=3$ and $N=1$. Left: $\sigma_2=\sigma_3=0$. Right: $\sigma_1=10$, $\sigma_3=0$. solid: $\pi_{{\rm MSVS1}}$ with $\gamma=p^2+p-2$, dashed: $\pi_{{\rm SVS}}$, dotted: Stein's prior $\pi_{{\rm S}}$.}
	\label{fig_MSVS1_pred}
\end{figure}
 
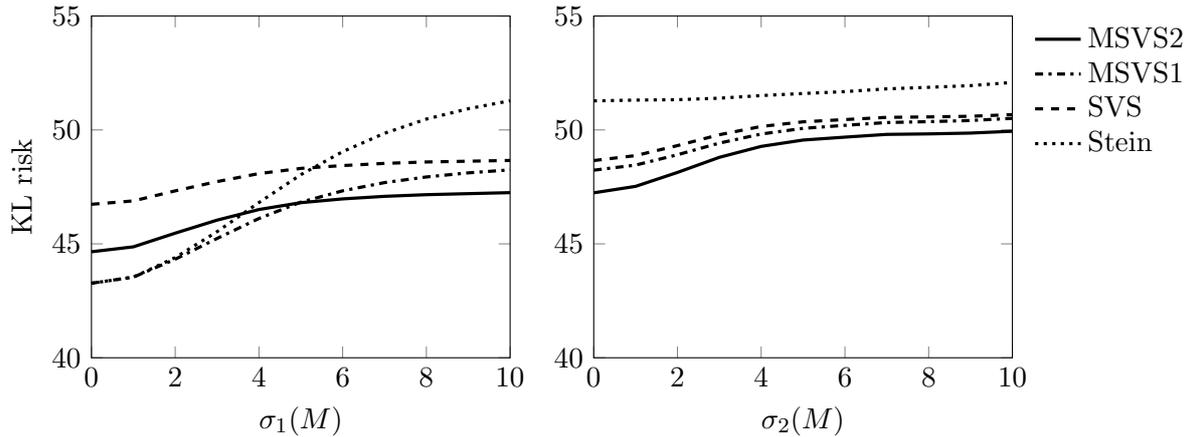
\begin{figure}[htbp]
	\centering
\begin{tikzpicture}
\tikzstyle{every node}=[]
	\begin{axis}[
		width=0.45\linewidth,
		xmax=10,xmin=0,
		ymax=55, ymin = 40,
xlabel={$\sigma_1(M)$},ylabel={KL risk},
ylabel near ticks,
		]
\addplot[very thick,
		filter discard warning=false, unbounded coords=discard
		] table {
         0   44.6530
1.0000   44.8694
2.0000   45.4691
3.0000   46.0414
4.0000   46.5070
5.0000   46.8082
6.0000   46.9733
7.0000   47.0857
8.0000   47.1592
9.0000   47.2040
10.0000   47.2487
		};
		\addplot[very thick, dash dot,
filter discard warning=false, unbounded coords=discard
] table {
         0   43.2723
1.0000   43.5388
2.0000   44.3249
3.0000   45.2368
4.0000   46.1135
5.0000   46.8304
6.0000   47.3285
7.0000   47.6891
8.0000   47.9364
9.0000   48.1168
10.0000   48.2475
};
		\addplot[very thick, dashed,
		filter discard warning=false, unbounded coords=discard
		] table {
         0   46.7366
1.0000   46.8852
2.0000   47.3257
3.0000   47.7332
4.0000   48.0839
5.0000   48.3121
6.0000   48.4365
7.0000   48.5318
8.0000   48.5946
9.0000   48.6277
10.0000   48.6616
		};
		\addplot[very thick, dotted,
		filter discard warning=false, unbounded coords=discard
		] table {
         0   43.2691
1.0000   43.5403
2.0000   44.3982
3.0000   45.5341
4.0000   46.8124
5.0000   48.0351
6.0000   49.0468
7.0000   49.8543
8.0000   50.4771
9.0000   50.9321
10.0000   51.2809
		};
	\end{axis}
\end{tikzpicture} 
\begin{tikzpicture}
\tikzstyle{every node}=[]
	\begin{axis}[
		width=0.45\linewidth,
		xmax=10,xmin=0,
		ymax=55, ymin = 40,
xlabel={$\sigma_2(M)$},ylabel near ticks,
		legend pos=outer north east,
		legend entries={MSVS2,MSVS1,SVS,Stein},
		legend style={legend cell align=left,draw=none,fill=white,fill opacity=0.8,text opacity=1,},
		]
\addplot[very thick,
		filter discard warning=false, unbounded coords=discard
		] table {
         0   47.2404
1.0000   47.5242
2.0000   48.1350
3.0000   48.7976
4.0000   49.2783
5.0000   49.5568
6.0000   49.6848
7.0000   49.8056
8.0000   49.8269
9.0000   49.8630
10.0000   49.9430
		};
		\addplot[very thick, dash dot,
filter discard warning=false, unbounded coords=discard
] table {
         0   48.2334
1.0000   48.4574
2.0000   48.9118
3.0000   49.4227
4.0000   49.8216
5.0000   50.0692
6.0000   50.2025
7.0000   50.3256
8.0000   50.3642
9.0000   50.4143
10.0000   50.5094
};
		\addplot[very thick, dashed,
		filter discard warning=false, unbounded coords=discard
		] table {
         0   48.6571
1.0000   48.8745
2.0000   49.3126
3.0000   49.7901
4.0000   50.1531
5.0000   50.3569
6.0000   50.4559
7.0000   50.5527
8.0000   50.5725
9.0000   50.6006
10.0000   50.6734
		};
		\addplot[very thick, dotted,
		filter discard warning=false, unbounded coords=discard
		] table {
         0   51.2755
1.0000   51.3107
2.0000   51.3279
3.0000   51.3931
4.0000   51.5088
5.0000   51.5996
6.0000   51.6823
7.0000   51.8037
8.0000   51.8708
9.0000   51.9453
10.0000   52.0850
		};
	\end{axis}
\end{tikzpicture} 
	\caption{Kullback--Leibler risk of Bayesian predictive densities for $n=10$, $p=3$ and $N=1$ where $M=U \Sigma$ with $U^{\top} U=I_p$ and $\Sigma={\rm diag}(\sigma_1,\dots,\sigma_p)$. Left: $\sigma_2=\sigma_3=0$. Right: $\sigma_1=10$, $\sigma_3=0$. solid: $\pi_{{\rm MSVS2}}$ with $\gamma_1=\dots=\gamma_p=p-1$, dash-dotted: $\pi_{{\rm MSVS1}}$ with $\gamma=p^2+p-2$, dashed: $\pi_{{\rm SVS}}$, dotted: Stein's prior $\pi_{{\rm S}}$.}
	\label{fig_MSVS2_pred}
\end{figure}

\section*{Acknowledgments}
	We thank the referees for helpful comments.
	This work was supported by JSPS KAKENHI Grant Numbers 19K20220, 21H05205, 22K17865 and JST Moonshot Grant Number JPMJMS2024.

\appendix

\section{Asymptotic expansion of risk}
Here, we provide asymptotic expansion formulas for estimators of a normal mean vector.
Consider the problem of estimating $\theta$ from the observation $Y \sim {\rm N}_d ( \theta, N^{-1} I_d )$ under the quadratic loss $l(\theta,\hat{\theta})=\| \hat{\theta}-\theta \|^2$.
As shown in \cite{Stein74}, the generalized Bayes estimator $\hat{\theta}^{\pi}$ with respect to a prior $\pi(\theta)$ is expressed as
\begin{align*}
	\hat{\theta}^{\pi} (y) = y + \frac{1}{N} \nabla_{y} \log m_{\pi} (y),
\end{align*}
where
\begin{align*}
	m_{\pi} (y) = \int p(y \mid \theta) \pi(\theta) {\rm d} \theta.
\end{align*}
The asymptotic difference between the quadratic risk of two generalized Bayes estimators as $N \to \infty$ is given as follows.

\begin{lemma}\label{lem_vec_est}
	As $N \to \infty$, the difference between the quadratic risk of $\hat{\theta}^{\pi_1}$ and $\hat{\theta}^{\pi_1 \pi_2}$ is expanded as
	\begin{align}
		& {\rm E}_{\theta} [\| \hat{\theta}^{\pi_1 \pi_2}-\theta \|^2] - {\rm E}_{\theta} [\| \hat{\theta}^{\pi_1}-\theta \|^2] \nonumber \\
		= & \frac{1}{N^2} ( 2 \nabla \log \pi_1 (\theta)^{\top} \nabla \log \pi_2 (\theta) + \| \nabla \log \pi_2 (\theta) \|^2 + 2 \Delta \log \pi_2 (\theta) ) + o(N^{-2}). \label{est_risk_diff}
	\end{align}
\end{lemma}
\begin{proof}
	By using Stein's lemma \citep{shr_book} and $m_{\pi}(y)=\pi(y)+o(1)$ as $N \to \infty$, the quadratic risk of the generalized Bayes estimator $\hat{\theta}^{\pi}$ is calculated as
	\begin{align}
		&{\rm E}_{\theta} [\| \hat{\theta}^{\pi}(y)-\theta \|^2] \nonumber \\
		=& {\rm E}_{\theta} [\| y - \theta \|^2] + \frac{2}{N} {\rm E}_{\theta} [(y - \theta)^{\top} \nabla \log m_{\pi} (y)] + \frac{1}{N^2} {\rm E}_{\theta} [\| \nabla \log m_{\pi} (y) \|^2] \nonumber \\
		=& \frac{d}{N} + \frac{1}{N^2} {\rm E}_\theta [ \| \nabla \log m_{\pi}(y) \|^2 + 2 \Delta \log m_{\pi} (y) ] \nonumber \\
		=& \frac{d}{N} + \frac{1}{N^2} \left( \| \nabla \log \pi (\theta) \|^2 + 2 \Delta \log \pi (\theta) \right) + o(N^{-2}). \label{est_asymp_risk}
	\end{align}
	Substituting $\pi=\pi_1 \pi_2$ and $\pi=\pi_1$ into \eqref{est_asymp_risk} and taking difference, we obtain \eqref{est_risk_diff}.
\end{proof}

We extend the above formula to matrices by using the matrix derivative notations from \cite{Matsuda22}.
For a function $f: \mathbb{R}^{n \times p} \to \mathbb{R}$, its {matrix gradient} $\widetilde{\nabla} f: \mathbb{R}^{n \times p} \to \mathbb{R}^{n \times p}$ is defined as 
\begin{align}
	(\widetilde{\nabla} f(X))_{ai} = \frac{\partial}{\partial X_{ai}} f(X). \label{matrix_grad}
\end{align}
For a $C^2$ function $f: \mathbb{R}^{n \times p} \to \mathbb{R}$, its {matrix Laplacian} $\widetilde{\Delta} f: \mathbb{R}^{n \times p} \to \mathbb{R}^{p \times p}$ is defined as 
\begin{align}
	(\widetilde{\Delta} f(X))_{ij} = \sum_{a=1}^n \frac{\partial^2}{\partial X_{ai} \partial X_{aj}} f(X).  \label{matrix_lap}
\end{align}
Then, the above formulas can be straightforwardly extended to matrix-variate normal distributions as follows.

\begin{lemma}\label{lem_mat_est}
	As $N \to \infty$, the difference between the Frobenius risk of $\hat{M}^{\pi_1}$ and $\hat{M}^{\pi_1 \pi_2}$ is expanded as
	\begin{align*}
		& {\rm E}_{M} [ \| \hat{M}^{\pi_1 \pi_2}-M \|_{\mathrm{F}}^2 ] - {\rm E}_{M} [ \| \hat{M}^{\pi_1}-M \|_{\mathrm{F}}^2 ] \nonumber \\
		= & \frac{1}{N^2} {\rm tr} ( 2 \widetilde{\nabla} \log \pi_1 (M)^{\top} \widetilde{\nabla} \log \pi_2 (M) + \widetilde{\nabla} \log \pi_2 (M)^{\top} \widetilde{\nabla} \log \pi_2 (M) + 2 \widetilde{\Delta} \log \pi_2 (M) ) \nonumber \\
		& \quad + o(N^{-2}).
	\end{align*}
\end{lemma}

\begin{lemma}\label{lem_mat_est2}
	As $N \to \infty$, the difference between the matrix quadratic risk of $\hat{M}^{\pi_1}$ and $\hat{M}^{\pi_1 \pi_2}$ is expanded as
	\begin{align*}
		& {\rm E}_{M} [ (\hat{M}^{\pi_1 \pi_2}-M)^{\top} (\hat{M}^{\pi_1 \pi_2}-M) - (\hat{M}^{\pi_1}-M)^{\top} (\hat{M}^{\pi_1}-M) ] \nonumber \\
		= & \frac{1}{N^2} ( 2 \widetilde{\nabla} \log \pi_1 (M)^{\top} \widetilde{\nabla} \log \pi_2 (M) + \widetilde{\nabla} \log \pi_2 (M)^{\top} \widetilde{\nabla} \log \pi_2 (M) + 2 \widetilde{\Delta} \log \pi_2 (M) ) \nonumber \\
		& \quad + o(N^{-2}).
	\end{align*}
\end{lemma}

\cite{Komaki06} derived the asymptotic expansion of the Kullback--Leibler risk of Bayesian predictive densities.
For the normal model as discussed in Section~\ref{sec_pred}, the result shows that Stein's prior dominates the Jeffreys prior in $O(N^{-1})$ term at the origin and $O(N^{-2})$ term at other points, which is reminiscent of superefficiency theory.
A similar phenomenon should exist in estimation as well.
Unlike Stein's prior, the priors for a normal mean matrix such as $\pi_{\mathrm{SVS}}$ diverge at many points such as low-rank matrices.
It is an interesting future problem to investigate the asymptotic risk of such priors in detail.

\section{Laplacian of $\pi_{{\rm MSVS1}}$ and $\pi_{{\rm MSVS2}}$}

\begin{lemma}\label{lem_laplacian}\citep{Stein74,Matsuda19}
	Suppose that $f:\mathbb{R}^{n \times p} \to \mathbb{R}$ is represented as $f(X) = \widetilde{f}(\sigma)$, where $n \geq p$ and $\sigma=(\sigma_1(X),\ldots,\sigma_p(X))$ denotes the singular values of $X$. If $f$ is twice weakly differentiable, then its Laplacian is
	\begin{align*}
		\Delta f &= \sum_{a=1}^n \sum_{i=1}^p \frac{\partial^2 f}{\partial X_{ai}^2} = 2 \sum_{i<j} \frac{\sigma_i {\partial \widetilde{f}}/{\partial \sigma_i} - \sigma_j  {\partial \widetilde{f}}/{\partial \sigma_j}}{\sigma_i^2 - \sigma_j^2} + (n-p) \sum_{i=1}^p \frac{1}{\sigma_i} \frac{\partial \widetilde{f}}{\partial \sigma_i} + \sum_{i=1}^p \frac{\partial^2 \widetilde{f}}{\partial \sigma_i^2}.
	\end{align*}
\end{lemma}

\begin{proposition}\label{prop_MSVS1_lap}
	The Laplacian of $\pi_{{\rm MSVS1}}$ in \eqref{MSVS1} is given by
	\begin{align*}
		\Delta \pi_{{\rm MSVS1}} (M) &= \gamma (\gamma+np-2p^2-2p+2) \| M \|_{{\rm F}}^{-2} \pi_{{\rm MSVS1}} (M).
	\end{align*}
\end{proposition}
\begin{proof}
	Let
	\begin{align*}
		\widetilde{f} (\sigma) = \left( \prod_{i=1}^p \sigma_i^{-(n-p-1)} \right) \left(\sum_{i=1}^p \sigma_i^2 \right)^{-\gamma/2}
	\end{align*}
	so that $\pi_{{\rm MSVS1}} (M)=\widetilde{f}(\sigma)$ with $\sigma=(\sigma_1(M),\dots,\sigma_p(M))$.
	From Lemma~\ref{lem_laplacian} and
	\begin{align*}
		\frac{\partial \widetilde{f}}{\partial \sigma_i} = -(n-p-1)\sigma_i^{-1} \widetilde{f} - \gamma \sigma_i \left(\sum_{j=1}^p \sigma_j^2 \right)^{-1} \widetilde{f},
	\end{align*}
	\begin{align*}
		\frac{\partial^2 \widetilde{f}}{\partial \sigma_i^2} = (n-p)(n-p-1)\sigma_i^{-2} \widetilde{f} + (2n-2p-3) \gamma \left(\sum_{j=1}^p \sigma_j^2 \right)^{-1} \widetilde{f} + \gamma (\gamma+2) \sigma_i^2 \left(\sum_{j=1}^p \sigma_j^2 \right)^{-2} \widetilde{f},
	\end{align*}
	we have
	\begin{align*}
		\Delta \pi_{{\rm MSVS1}}(M) &= 2 \sum_{i<j} \frac{\sigma_i {\partial \widetilde{f}}/{\partial \sigma_i} - \sigma_j  {\partial \widetilde{f}}/{\partial \sigma_j}}{\sigma_i^2 - \sigma_j^2} + (n-p) \sum_{i=1}^p \frac{1}{\sigma_i} \frac{\partial \widetilde{f}}{\partial \sigma_i} + \sum_{i=1}^p \frac{\partial^2 \widetilde{f}}{\partial \sigma_i^2} \\
		&= -2 \cdot \frac{p(p-1)}{2} \gamma \left(\sum_{i=1}^p \sigma_i^2 \right)^{-1} \widetilde{f} - (n-p) (n-p-1) \left( \sum_{i=1}^p \sigma_i^{-2} \right) \widetilde{f} \\
		& \quad - p(n-p) \gamma \left(\sum_{i=1}^p \sigma_i^2 \right)^{-1} \widetilde{f} + (n-p) (n-p-1) \left( \sum_{i=1}^p \sigma_i^{-2} \right) \widetilde{f} \\
		& \quad + \gamma (p(2n-2p-3) + \gamma+2) \left(\sum_{i=1}^p \sigma_i^2 \right)^{-1} \widetilde{f}\\
		& = \gamma (\gamma+np-2p^2-2p+2) \left(\sum_{i=1}^p \sigma_i^2 \right)^{-1} \widetilde{f}.
	\end{align*}
\end{proof}

\begin{proposition}\label{prop_MSVS2_lap}
	The Laplacian of $\pi_{{\rm MSVS2}}$ in \eqref{MSVS2} is given by
	\begin{align*}
		\Delta \pi_{{\rm MSVS2}} (M) &= \gamma (\gamma+n-2p) \left( \sum_{i=1}^p \| M_{\cdot i} \|^{-2} \right) \pi_{{\rm MSVS2}} (M).
	\end{align*}
\end{proposition}
\begin{proof}
	From
	\begin{align*}
		\Delta \log f = \frac{\Delta f}{f} - \| \nabla \log f \|^2
	\end{align*}
	and \eqref{nabla6}, \eqref{nabla7} and \eqref{nabla8},
	\begin{align*}
		\frac{\Delta \pi_{{\rm MSVS2}} (M)}{\pi_{{\rm MSVS2}} (M)} &= \Delta \log \pi_{{\rm MSVS2}}(M) + \| \nabla \log \pi_{{\rm MSVS2}}(M) \|^2 \\
		&= \Delta \log \pi_{{\rm SVS}}(M) + \Delta \log \pi_{{\rm CS}}(M) + \| \nabla \log \pi_{{\rm SVS}}(M) + \nabla \log \pi_{{\rm CS}}(M) \|^2 \\
		&= \frac{\Delta \pi_{{\rm SVS}}(M)}{\pi_{{\rm SVS}}(M)} +(-(n-2)\gamma + 2(n-p-1)\gamma + \gamma^2) \sum_{i=1}^p  \| M_{\cdot i} \|^{-2} \\
		&= \gamma (\gamma+n-2p) \sum_{i=1}^p \| M_{\cdot i} \|^{-2},
	\end{align*}
	where we used $\Delta \pi_{{\rm SVS}}(M)=0$ \citep[Theorem 2 of][]{Matsuda}.
\end{proof}

\section{Improving on the block-wise Stein prior}
Here, we develop priors that asymptotically dominate the block-wise Stein prior in estimation and prediction.
Suppose that we observe $Y \sim {\rm N}_d (\theta, N^{-1} I_d)$ and estimate $\theta$ or predict $\widetilde{Y} \sim {\rm N}_d (\theta, I_d)$. 
We assume that the $d$-dimensional mean vector $\theta$ is split into $B$ disjoint blocks $\theta^{(1)},\dots,\theta^{(B)}$ with size $d_1,\dots,d_B$, where $d_1+\dots+d_B=d$.
For example. such a situation appears in balanced ANOVA and wavelet regression \citep{Brown}.
Then, the block-wise Stein prior is defined as
\begin{align}
	\pi_{{\rm BS}} (\theta) = \prod_{b=1}^B \| \theta^{(b)} \|^{R_b}, \quad R_b=-(d_b-2)_{+}, \label{BS_prior}
\end{align}
which puts Stein's prior on each block.
Since it is superharmonic, the generalized Bayes estimator $\hat{\theta}^{\pi_{{\rm BS}}}$ with respect to $\pi_{{\rm BS}}$ is minimax.
However, \cite{Brown} showed that $\hat{\theta}^{\pi_{{\rm BS}}}$ is inadmissible and dominated by an estimator with additional James--Stein type shrinkage defined by
\begin{align*}
	\hat{\theta} (y) = \hat{\theta}^{\pi_{{\rm BS}}} (y) - \frac{R_{\#}+d-2}{\| y \|^2} y,
\end{align*}
where $R_{\#} = \sum_b R_b>2-d$.
From this result, \cite{Brown} conjectured that the block-wise Stein prior can be improved by multiplying a Stein-type shrinkage prior in Remark 3.2.
Following their conjecture, we construct priors by adding scalar shrinkage to the block-wise Stein priors:
\begin{align}
	\pi_{{\rm MBS}} (\theta) = \pi_{{\rm BS}} (\theta) \| \theta \|^{-\gamma}, \label{MBS_prior}
\end{align}
where $\gamma \geq 0$.
Let
\begin{align*}
	m_{{\rm MBS}} (y) = \int p(y \mid \theta) \pi_{{\rm MBS}} (\theta) {\rm d} \theta.
\end{align*}

\begin{lemma}\label{lem_MBS}
	If $0 \leq \gamma < B(R_b+d_b)$ for every $b$, then $m_{{\rm MBS}} (y) < \infty$ for every $y$.
\end{lemma}
\begin{proof}
	Since $m_{{\rm MBS}} (y)$ is interpreted as the expectation of $\pi_{{\rm MBS}} (\theta)$ under $\theta \sim {\rm N}_d(y,I_d)$, it suffices to show that $\pi_{{\rm MBS}} (\theta)$ is locally integrable at every $\theta$.
	
	First, consider $\theta \neq 0$.
	Since $m_{{\rm BS}} (y) < \infty$ for every $y$ \citep{Brown}, $\pi_{{\rm BS}} (\theta)$ is locally integrable at $\theta$.
	Also, $\| \theta \|>c$ for some $c>0$ in a neighborhood of $\theta$.
	Thus, $\pi_{{\rm MBS}} (\theta)=\pi_{{\rm BS}} (\theta) \| \theta \|^{-\gamma}$ is locally integrable at $\theta$.
	
	Next, consider $\theta = 0$ and take the neighborhood $A=\{ \theta \mid \| \theta^{(1)} \| \leq s, \dots, \| \theta^{(B)} \| \leq s \}$ for $s>0$.
	From the AM-GM inequality,
	\begin{align*}
		\| \theta \|^2 = \sum_b \| \theta^{(b)} \|^2 \geq B \left( \prod_b \| \theta^{(b)} \|^2 \right)^{1/B} = B \prod_b \| \theta^{(b)} \|^{2/B}.
	\end{align*}
	Thus,
	\begin{align*}
		\int_A \pi_{{\rm MBS}} (\theta) {\rm d} \theta \leq \ & C \int_0^{s} \cdots \int_0^{s} \prod_b r_b^{R_b+d_b-1-\gamma/B} {\rm d} r_1 \cdots {\rm d} r_B,
	\end{align*}
	where $r_b=\| \theta^{(b)} \|$ and $C$ is a constant.
	Therefore, $\pi_{{\rm MBS}} (\theta)$ is locally integrable at $\theta=0$ if $R_b+d_b-1-\gamma/B>-1$ for every $b$, which is equivalent to $\gamma < B(R_b+d_b)$ for every $b$.
\end{proof}

From Lemma~\ref{lem_MBS}, the generalized Bayes estimator with respect to $\pi_{{\rm MBS}}$ is well-defined when $0 \leq \gamma < B(R_b+d_b)$ for every $b$.
We denote it by $\hat{\theta}_{{\rm MBS}}$.

\begin{theorem}\label{th_MBS_dom}
	For every $M$,
	\begin{align}
		N^2 ({\rm E}_{\theta} [\| \hat{\theta}_{{{\rm MBS}}}-\theta \|^2] - {\rm E}_{\theta} [\| \hat{\theta}_{{{\rm BS}}}-\theta \|^2]) \to \gamma (\gamma-2 (R_{\#}+d-2)) \| \theta \|^{-2} \label{risk_diff_BS}
	\end{align}
	as $N \to \infty$.
	Therefore, if $0 < \gamma < 2 (R_{\#}+d-2)$, then the generalized Bayes estimator with respect to $\pi_{{\rm MBS}}$ in \eqref{MBS_prior} asymptotically dominates that with respect to $\pi_{{\rm BS}}$ in \eqref{BS_prior} under the Frobenius loss.
\end{theorem}
\begin{proof}
	Let $\pi_{{\rm S}} (\theta) = \| \theta \|^{-\gamma}$.
	By straightforward calculation, we obtain
	\begin{align*}
		\nabla \log \pi_{{\rm BS}} (\theta)^{\top} \nabla \log \pi_{{\rm S}} (\theta) = -\gamma R_{\#} \| \theta \|^{-2},\\
		\nabla \log \pi_{{\rm S}} (\theta)^{\top} \nabla \log \pi_{{\rm S}} (\theta) = \gamma^2 \| \theta \|^{-2},\\
		\Delta \log \pi_{{\rm S}} (\theta) = -\gamma (d-2) \| \theta \|^{-2}.
	\end{align*}
	Therefore, from Lemma~\ref{lem_vec_est},
	\begin{align*}
		& {\rm E}_{\theta} [\| \hat{\theta}_{{{\rm MBS}}}-\theta \|^2] - {\rm E}_{\theta} [\| \hat{\theta}_{{{\rm BS}}}-\theta \|^2] \\
		= & \frac{1}{N^2} \left( 2 \nabla \log \pi_{{\rm BS}} (\theta)^{\top} \nabla \log \pi_{{\rm S}} (\theta) + \| \nabla \log \pi_{{\rm S}} (\theta) \|^2 + 2 \Delta \log \pi_{{\rm S}} (\theta) \right) + o(N^{-2}) \\
		= & \frac{1}{N^2} \gamma (\gamma-2 (R_{\#}+d-2)) \| \theta \|^{-2} + o(N^{-2}).
	\end{align*}
	Hence, we obtain \eqref{risk_diff_BS}.
\end{proof}

From \eqref{risk_diff_BS}, the choice $\gamma = R_{\#}+d-2$ is optimal.
As discussed in Section~\ref{sec_pred}, Theorem~\ref{th_MBS_dom} is extended to Bayesian prediction as follows.

\begin{theorem}
	For every $M$,
	\begin{align*}
		N^2 ({\rm E}_{\theta} [ D(p(\cdot \mid \theta),\hat{p}_{{\rm MBS}}(\cdot \mid y) ) ] - {\rm E}_{\theta} [ D(p(\cdot \mid \theta),\hat{p}_{{\rm BS}}(\cdot \mid y) ) ]) \to \frac{\gamma (\gamma-2 (R_{\#}+d-2))}{2} \| \theta \|^{-2}
	\end{align*}
	as $N \to \infty$.
	Therefore, if $p \geq 2$ and $0 < \gamma < p^2+p$, then the Bayesian predictive density with respect to $\pi_{{\rm MBS}}$ in \eqref{MBS_prior} asymptotically dominates that with respect to $\pi_{{\rm BS}}$ in \eqref{BS_prior} under the Kullback--Leibler loss.
\end{theorem}


\begin{thebibliography}{99}
	\expandafter\ifx\csname natexlab\endcsname\relax\def\natexlab#1{#1}\fi
	
	\bibitem[{Aitchison(1975)}]{Aitchison}
	\textsc{Aitchison, J.} (1975).
	\newblock{Goodness of prediction fit}.
	\newblock \textit{Biometrika} \textbf{62}, 547--554.
	
	\bibitem[{Amari(1985)}]{Amari}
	\textsc{Amari,~S.} (1985).
	\newblock \textit{Differential-Geometrical Methods in Statistics}.
	\newblock New York: Springer.
	
	\bibitem[{Brown(1971)}]{Brown71}
	\textsc{Brown, L. D.} (1971).
	\newblock{Admissible estimators, recurrent diffusions, and insoluble boundary value problems}.
	\newblock \textit{Annals of Mathematical Statistics} \textbf{36}, 855--903.
	

	\bibitem[{Brown and Zhao(2009)}]{Brown}
	\textsc{Brown, L. D.} \& \textsc{Zhao, L. H.} (2009).
	\newblock{Estimators for Gaussian models having a block-wise structure}.
	\newblock \textit{Statistica Sinica} \textbf{19}, 885--903.
	



	\bibitem[{Efron and Morris(1972)}]{Efron72}
	\textsc{Efron, B.} \& \textsc{Morris, C.} (1972).
	\newblock{Empirical Bayes on vector observations: an extension of Stein's method}.
	\newblock \textit{Biometrika} \textbf{59}, 335--347.
	
	\bibitem[{Efron and Morris(1976)}]{Efron76}
	\textsc{Efron, B.} \& \textsc{Morris, C.} (1976).
	\newblock{Multivariate empirical Bayes and estimation of covariance matrices}.
	\newblock \textit{Annals of Statistics} \textbf{4}, 22--32.
	
	\bibitem[{Fourdrinier et al.(2018)}]{shr_book}
	\textsc{Fourdrinier, D.}, \textsc{Strawderman, W. E.} \& \textsc{Wells, M.} (2018).
	\newblock \textit{Shrinkage Estimation}.
	\newblock Springer, New York.
	

	\bibitem[{George, Liang and Xu(2006)}]{George06}
	\textsc{George, E. I.}, \textsc{Liang, F.} \& \textsc{Xu, X.} (2006).
	\newblock{Improved minimax predictive densities under Kullback--Leibler loss}.
	\newblock \textit{Annals of Statistics} \textbf{34}, 78--91.
	

	\bibitem[{Gupta and Nagar(2000)}]{Gupta}
	\textsc{Gupta, A. K.} \& \textsc{Nagar, D. K.} (2000).
	\newblock \textit{Matrix Variate Distributions}.
	\newblock New York: Chapman \& Hall.
	



	\bibitem[{Komaki(2001)}]{Komaki01}
	\textsc{Komaki, F.} (2001).
	\newblock{A shrinkage predictive distribution for multivariate normal observables}.
	\newblock \textit{Biometrika} \textbf{88}, 859--864.
	
	\bibitem[{Komaki(2006)}]{Komaki06}
	\textsc{Komaki, F.} (2006).
	\newblock{Shrinkage priors for Bayesian prediction}.
	\newblock \textit{Annals of Statistics} \textbf{34}, 808--819.
	

	\bibitem[{Konno(1990)}]{Konno90}
	\textsc{Konno, Y.} (1990).
	\newblock{Families of minimax estimators of matrix of normal means with unknown covariance matrix}.
	\newblock \textit{J. Japan Statist. Soc.} \textbf{20}, 191--201.
	
	\bibitem[{Konno(1991)}]{Konno91}
	\textsc{Konno, Y.} (1991).
	\newblock{On estimation of a matrix of normal means with unknown covariance matrix}.
	\newblock \textit{Journal of Multivariate Analysis} \textbf{36}, 44--55.
	

	\bibitem[{Matsuda(2023)}]{Matsuda23a}
	\textsc{Matsuda, T.} (2023).
	\newblock{Adapting to arbitrary quadratic loss via singular value shrinkage}.
	\newblock \textit{IEEE Transactions on Information Theory}, accepted.
	
	\bibitem[{Matsuda(2024)}]{Matsuda23b}
	\textsc{Matsuda, T.} (2024).
	\newblock{Matrix quadratic risk of orthogonally invariant estimators for a normal mean matrix}.
	\newblock \textit{Japanese Journal of Statistics and Data Science}, accepted.
	
	\bibitem[{Matsuda and Komaki(2015)}]{Matsuda}
	\textsc{Matsuda, T.} \& \textsc{Komaki, F.} (2015).
	\newblock{Singular value shrinkage priors for Bayesian prediction}.
	\newblock \textit{Biometrika} \textbf{102}, 843--854.
	
	\bibitem[{Matsuda and Strawderman(2019)}]{Matsuda19}
	\textsc{Matsuda, T.} \& \textsc{Strawderman, W. E.} (2019).
	\newblock{Improved loss estimation for a normal mean matrix}.
	\newblock \textit{Journal of Multivariate Analysis} \textbf{169}, 300--311.
	
	\bibitem[{Matsuda and Strawderman(2022)}]{Matsuda22}
	\textsc{Matsuda, T.} \& \textsc{Strawderman, W. E.} (2022).
	\newblock{Estimation under matrix quadratic loss and matrix superharmonicity}.
	\textit{Biometrika} \textbf{109}, 503--519.
	


	\bibitem[{Stein(1974)}]{Stein74}
	\textsc{Stein, C.} (1974).
	\newblock{Estimation of the mean of a multivariate normal distribution}.
	\newblock \textit{Proc. Prague Symp. Asymptotic Statistics} \textbf{2}, 345--381.
	
	\bibitem[{Tsukuma and Kubokawa(2007)}]{Tsukuma07}
	\textsc{Tsukuma, H.} \& \textsc{Kubokawa, T.} (2007).
	\newblock{Methods for improvement in estimation of a normal mean matrix}.
	\newblock \textit{Journal of Multivariate Analysis} \textbf{98}, 1592--1610.
	

	\bibitem[{Tsukuma and Kubokawa(2017)}]{Tsukuma17}
	\textsc{Tsukuma, H.} \& \textsc{Kubokawa, T.} (2017).
	\newblock{Proper Bayes and minimax predictive densities related to estimation of a normal mean matrix}.
	\newblock \textit{Journal of Multivariate Analysis} \textbf{159}, 138--150.
	
	\bibitem[{Tsukuma and Kubokawa(2020)}]{Tsukuma}
	\textsc{Tsukuma, H.} \& \textsc{Kubokawa, T.} (2020).
	\textit{Shrinkage estimation for mean and covariance matrices}.
	Springer.
	
	\bibitem[{Yuasa and Kubokawa(2023a)}]{Yuasa1}
	\textsc{Yuasa, R.} \& \textsc{Kubokawa, T.} (2023).
	\newblock{Generalized Bayes estimators with closed forms for the normal mean and covariance matrices}.
	\newblock \textit{Journal of Statistical Planning and Inference} \textbf{222}, 182--194.
	
	\bibitem[{Yuasa and Kubokawa(2023b)}]{Yuasa12}
	\textsc{Yuasa, R.} \& \textsc{Kubokawa, T.} (2023).
	\newblock{Weighted shrinkage estimators of normal mean matrices and dominance properties}.
	\newblock \textit{Journal of Multivariate Analysis} \textbf{194}, 1--17.
\end{thebibliography}
\end{document}